\documentclass[11pt,letter]{amsart}
\usepackage{amssymb,amsmath,epsfig,graphics,mathrsfs,enumerate,verbatim}
\usepackage[pagebackref,colorlinks=true,linkcolor=blue,citecolor=blue]{hyperref}
\usepackage{fancyhdr}
\usepackage{hyperref}
\hypersetup{
 colorlinks   = true,
 urlcolor     = blue,
 linkcolor    = blue,
 citecolor   = red ,
 bookmarksopen=true
}

\usepackage{amsmath}
\usepackage{amsfonts}
\usepackage{amssymb}
\usepackage{amsthm}
\usepackage{epsfig,graphics,mathrsfs}
\usepackage{graphicx}
\usepackage{dsfont}

\usepackage[usenames, dvipsnames]{color}

\usepackage{hyperref}

\usepackage[a4paper,
bindingoffset=0in,
left=0.8in,
right=0.8in,
top=1.3in,
bottom=1.3in,
footskip=.25in]{geometry}

\def \O {\Omega}
\def \phi {\varphi}

\def \R {\mathbb{R}}

\def \G{\Gamma}

\def \vf{\varphi}

\def \So {\mathscr{S}(\Rn)}

\newcommand{\Rn}{\mathbb R^n}

\newcommand{\p}{\partial}

\newcommand{\bG}{\mathbb {G}}
\newcommand{\bg}{\mathfrak g}

\newcommand{\la}{\lambda}

\numberwithin{equation}{section}

\newcommand{\beq}{\begin{equation}}
\newcommand{\bea}[1]{\begin{array}{#1} }
\newcommand{\eeq}{ \end{equation}}
\newcommand{\ea}{ \end{array}}

\newcommand{\ve}{\varepsilon}

\newcommand{\I}{\mathscr I_{HL}}

\newcommand{\tr}{\operatorname{tr}}

\newcommand{\sa}{\langle}
\newcommand{\da}{\rangle}

\newtheorem{theorem}{Theorem}[section]
\newtheorem{lemma}[theorem]{Lemma}
\newtheorem{proposition}[theorem]{Proposition}
\newtheorem{corollary}[theorem]{Corollary}

\newtheorem{definition}[theorem]{Definition}
\newtheorem{example}[theorem]{Example}

\numberwithin{equation}{section}
\begin{document}

%%%%%%%%%%%%%%%%%%%%%%%%%%%%%%%%%%%%%%%%%%%%%%%%%%%%%%%%%%%%%%%%%%%%%%%%%%%%%%

\title[]{Strichartz estimates for a class of Schr\"odinger  equations with a drift}

	
	
\date{}
	
\begin{abstract}
We establish new intrinsic Strichartz estimates for solutions of the Cauchy problem for a class of possibly degenerate Schr\"odinger equations with a real drift. 
\end{abstract}
	
\author{Federico Buseghin}
	
\address{Laboratoire AGM\\ CY Cergy Paris Universit\'e\\ 2 avenue Adolphe Chauvin, 95300 Pontoise, France}
\vskip 0.2in
\email[Federico Buseghin]{federico.buseghin@cyu.fr}

\author{Nicola Garofalo}
\address{School of Mathematical and Statistical Sciences\\ Arizona State University}\email[Nicola Garofalo]{nicola.garofalo@asu.edu}

\thanks{The work of Federico Buseghin is part of the ERC Starting Grant project FloWAS that has received funding from the European Research Council (ERC) under the European Union's Horizon 2020 research and innovation program (Grant agreement No. 101117820).}

\maketitle
	
\tableofcontents

\section{Introduction and statement of the main results}\label{S:intro}

In his seminal paper \cite[Cor.1]{Stri} Strichartz proved that, if $u$ solves the Cauchy problem in $\Rn\times [0,\infty)$
\begin{equation}\label{CPs}
\p_t u - i \Delta u = F(x,t),\ \ \ \ \ \ \ u(x,0) = \vf(x),
\end{equation}
then there exists a constant $C(n)>0$ such that
\begin{equation}\label{stri}
||u||_{L^{\frac{2(n+2)}n}(\R^{n+1})} \le C(n) \left(||\vf||_{L^2(\Rn)} + ||F||_{L^{\frac{2(n+2)}{n+4}}(\R^{n+1})}\right).
\end{equation}
It is well-known that estimates such as \eqref{stri} play a basic role in the analysis of nonlinear Schr\"odinger equations, see \cite{Caze} and \cite{Tao}. Strichartz established \eqref{stri} by generalizing to a paraboloid in $\R^{n+1}$ the Tomas-Stein restriction theorem for the sphere in $\Rn$. In this context, we note that the exponent $\frac{2(n+2)}{n+4}$ on the right-hand side of \eqref{stri} is the restriction exponent in $\R^{n+1}$, while $\frac{2(n+2)}{n}$ on the left-hand side is its dual. In turn, the (dual of the) estimate \eqref{stri} implies the restriction inequality for the paraboloid, see \cite[pp.100-101]{Sogge}, where a similar argument is presented for the cone. Remarkably, such restriction theorem can also be independently established by the ingenious approach of Ginibre-Velo, see the remark in Section \ref{S:res}.    

In this paper, we establish some new mixed norm estimates for solutions of a class of dispersive equations with a real drift. To introduce the relevant framework, for $n\ge 2$ let $Q, B \in \mathbb M_{n\times n}(\R)$ be two constant matrices, with $Q = Q^\star$ and $Q\ge 0$. For suitable data $\vf$ and $F$, we consider the Cauchy problem  in $\Rn\times [0,\infty)$
\begin{equation}\label{A}
\p_t u - i \tr(Q\nabla^2 u) -  \sa Bx,\nabla u\da   = F(x,t),\ \ \ \ \ \ \ u(x,0) = \vf(x).
\end{equation}
Schr\"odinger equations of the form \eqref{A} naturally arise in various contexts of both physical and mathematical significance, see Section \ref{S:mot} for some motivation and background. Here, we highlight three key features of \eqref{A}:

\begin{itemize} \item[(i)] The matrix $Q$ can be highly degenerate, with even the possibility of $\operatorname{Rank}(Q) =1$ (see Proposition \ref{E:fan}). This is a typical situation in the modeling of quantum wires -- one-dimensional nanostructures where electrons are confined to two dimensions but free to move in the third. When subjected to a magnetic field, the Hamiltonian governing these systems exhibits degenerate kinetic energy, alongside a drift term that shifts the dynamics in a specific direction.

\item[(ii)] Another challenge is that, due to the drift, the PDE in \eqref{A} lacks a variational structure.

\item[(iii)] A third feature, closely related to the drift, is the general absence of a global homogeneous structure (dilations). 
\end{itemize}

\vskip 0.2in

\noindent \textbf{Main assumption.} The following algebraic condition will be assumed throughout the remainder of this paper:

\vskip 0.2in

\noindent \ \ \  \textbf{(H)}\ \ \ \ \ \ $\operatorname{Ker} Q$\ \text{\emph{does not contain any nontrivial invariant subspace of}}\ $B^\star$.

\vskip 0.2in

\noindent This hypothesis is equivalent to the strict positivity of the \emph{controllability Gramian} of the matrices $Q$ and $B$ defined by 
\begin{equation}\label{Ds}
Q(t) = \int_0^t e^{sB}Q e^{sB^\star} ds,
\end{equation} 
see  \cite[Theor.1.2, p.17]{Z}. We now recall that in \cite[Prop.2.4]{GL} it has been recently shown that, when $F = 0$, the problem \eqref{A} admits the following solution 
\begin{equation}\label{erf}
u(x,t)  = 
 \frac{(4\pi
)^{-\frac{n}{2}}e^{-\frac{i \pi n}4}}{\sqrt{\det Q(t)}}   \int_{\Rn} e^{i \frac{\sa Q(t)^{-1}(e^{tB}x-y),e^{tB}x - y\da}{4}} \vf(y) dy,\ \ \ \ t>0. 
\end{equation}
The goal of this work is to use the representation \eqref{erf}  as a starting point for a generalized theory of Strichartz estimates. We aim to show that the broader framework provided by the class of equations \eqref{A} leads to the identification of new a priori inequalities, where the exponents of integrability of the solution are determined exclusively by the \emph{local homogeneous dimension} $D$ of a nilpotent Lie algebra, intrinsically defined by the matrices $Q$ and $B$.
While this idea shares similarities with the works \cite{RS} and \cite{NSW}, it introduces a novel perspective in the context of Strichartz estimates for Schr\"odinger equations such as those in \eqref{A}, where the associated Hamiltonian may exhibit significant degeneracies. 

\noindent Formula \eqref{erf} highlights the central role of the function
\begin{equation}\label{V}
t\longrightarrow V(t) \overset{def}{=} \det Q(t),
\end{equation}
in the analysis of the Schr\"odinger operators \eqref{A}. Since when $Q = I_n$ and $B = O_n$, we have $V(t) = t^n$, it is evident that this function can be thought of as a measure of volume. 
For general $Q$ and $B$, the function $V(t)$ offers valuable insights into how the geometry of the problem evolves with time, particularly in the study of dispersive effects and long-time asymptotics. We note in this respect that, if we consider the time-varying, tilted ellipsoids 
\[
\mathscr E(x,t) = \{y\in \Rn\mid \langle Q(t)^{-1}(e^{tB} x - y),e^{tB} x -y\rangle <  1\}
\]
in the imaginary Gaussian in \eqref{erf},
it is easy to recognize that $V(t)$ is connected to the $n$-dimensional volume of $\mathscr E(x,t)$  by the equation
\begin{equation}\label{vol}
V(t) = \omega_n^{-2} \operatorname{Vol}_n(\mathscr E(x,t))^2.
\end{equation}
For these reasons, we will henceforth refer to $V(t)$ in \eqref{V} as the \emph{volume function}. 

\vskip 0.2in

%%%%%%%%%%%%%%%%%%%%%%%%%%%%%%%%%%%%%%%%%%%%%%%%%%%

\noindent To continue our discussion, we recall that in the paper \cite{LP} it was shown that  \textbf{(H)} is equivalent to having $Q$ and $B$ in the following \emph{canonical form}
\begin{equation}\label{QB}
Q = \begin{pmatrix} Q_0 & O_{p_0\times{(n-p_0)}}
\\
O_{{(n-p_0)}\times p_0} & O_{(n-p_0)\times (n-p_0)}\end{pmatrix},
\ \ \ \ \ \ \ \ 
B = \begin{pmatrix} \star & \star & \cdot & \cdot & \cdot & \star & \star
\\
B_1 & \star & \cdot & \cdot & \cdot & \star & \star
\\
0 & B_2 & \star & \cdot & \cdot &  \star & \star
\\
\cdot & \cdot & \cdot & \cdot & \cdot & \cdot & \cdot
\\
\cdot & \cdot & \cdot & \cdot & \cdot & \cdot & \cdot
\\
\cdot & \cdot & \cdot & \cdot & \cdot & \cdot & \cdot
\\
0 & 0 & \cdot & \cdot & \cdot &  B_r & \star
\end{pmatrix},
\end{equation}
where $Q_0 = Q_0^\star$ is a strictly positive $p_0\times p_0$ matrix, and for $j=1,...,r$, $B_j$ is a $ p_j \times p_{j-1}$ matrix such that $\operatorname{Rank}(B_j) = p_j$,  with 
\[
p_0\ge p_1\ge ... \ge p_r \ge 1,\ \ \  \text{and}\ \ \ p_0 + p_1 + ... + p_r = n.
\]
The matrices designated by a $\star$ in the above expression of $B$ can be arbitrary. When $Q$ itself is invertible, and therefore $p_0 = n$,  we have $p_j = 0$ for $j=1,...,r$, and $B$ coincides with the $\star$ in the upper left corner. We remark that $B$ in \eqref{QB} need not be nilpotent (see for instance the example \eqref{ex} below). Another important point is that the canonical form \eqref{QB} must be interpreted dynamically. Specifically, the canonical representation of $B$, and consequently the range of the Strichartz estimates, changes when $Q$ is altered. This phenomenon is clearly illustrated in Proposition \ref{E:fan}.

\medskip

\noindent Throughout the rest of the paper, we assume that $Q$ and $B$ in \eqref{A} are in the canonical form \eqref{QB}. 

\medskip

\begin{definition}\label{D:hd}
We call the number  
\begin{equation}\label{hd}
D = p_0+3p_1+\ldots+(2r+1)p_r
\end{equation}
the \emph{local homogeneous dimension} of the Schr\"odinger operator \eqref{A}. 
\end{definition}

\medskip

While the motivation for Definition \ref{D:hd} will be clarified in Section \ref{S:prelim}, we note here that the number $D$ will play a pervasive role in our results. Observe that \eqref{hd} implies that:
\begin{itemize}  
\item[(a)] $D\ge p_0+p_1+\ldots+p_r = n\ge 2$;
\item[(b)] If $Q$ is invertible, we take $p_j=0$ for $j=1,...,r$, and therefore $D =n= p_0$;
\item[(c)] When $Q$ is not invertible, we have $D>n$. This immediately follows from \eqref{hd} since $p_0<n$;
\item[(d)] An elementary dimension counting argument shows that the maximum value of $D$ is given by 
\begin{equation}\label{Dmax}
D= \sum_{k=1}^{n} (2k - 1) = n^2,
\end{equation}
\end{itemize}
see also \eqref{Dmaxx}.

\vskip 0.2in

\noindent To illustrate Definition \ref{D:hd}, consider for instance the degenerate Schr\"odinger equation in $\R^3\times \R$
\begin{equation*}
\p_t u = i (\p_{xx} u + \p_{yy} u) + x \p_x u + y \p_z u.
\end{equation*}
The corresponding matrices are in canonical form
\begin{equation}\label{ex}
Q = \begin{pmatrix} 1 & 0 & 0 \\ 0 & 1 & 0\\ 0 & 0 & 0 \end{pmatrix},\ \ \ B = \begin{pmatrix} 1 & 0 & 0 \\ 0 & 0 & 0\\ 0 & 1 & 0 \end{pmatrix} = \begin{pmatrix} \star_1 & \star_2\\ B_1& \star_3\end{pmatrix},
\end{equation}
with 
\[
Q_0 = \begin{pmatrix} 1 & 0 \\ 0 & 1\end{pmatrix},\ \ \star_1 = \begin{pmatrix} 1 & 0 \\ 0 & 0\end{pmatrix}, \ \star_2 \begin{pmatrix} 0 \\ 0\end{pmatrix},\ \ B_1 = (0\ 1),\ \ \star_3 = (0).
\]
Note that $B$ is not nilpotent and therefore, based on the results in \cite{Ka2}, \cite{LP}, there can be no global dilations attached to the PDE. 
However, following Definition \ref{D:hd}, it is possible to assign a local homogeneous dimension $D$. Since $p_0 = \operatorname{Rank}(Q_0)= 2$ and $p_1 = \operatorname{Rank}(B_1)= 1$, we find that $D = p_0 + 3 p_1 = 5$. 
This number will determine the integrability range in the Strichartz estimate which results from an application of Theorem A below. Additionally, an exponential weight will arise from the contribution of the eigenvalue $\la =1$ of $B$, but this will be explained in more detail later.

\medskip

\noindent From now on, the letter $D$ will be used exclusively to represent the number defined by \eqref{hd}. 

\begin{definition}\label{D:admissible}
Given a number $r$ such that $2\le r < \infty$ when $D = 2$, or 
$2\le r \le \frac{2D}{D-2}$ when $D> 2$, 
we say that the pair $(q,r)$ is \emph{admissible} for \eqref{A} if one has
\begin{equation}\label{admissible}
\frac 2q = D\left(\frac 12 - \frac 1r\right).
\end{equation}
\end{definition}

\medskip

\noindent We note that in \eqref{admissible} we can take 
 \begin{equation}\label{rr}
q = r = \frac{2(D+2)}D,\ \ \ \ \ \ \ q' = r' = \frac{2(D+2)}{D+4}.
\end{equation}
This follows from observing that the pair $(r,r)$ is admissible for \eqref{A}, provided that $r$ is given by \eqref{rr}. Note that, for such $r$, we have $r>2$, and that $r<\frac{2D}{D-2}$ when $D>2$. We call the numbers $r, r'$ in \eqref{rr} the \emph{Strichartz pairs} relative to $D$. The motivation for this name is that, as noted in (b) above, when $Q$ in \eqref{QB} is invertible, we have $D = n$, and the corresponding pair is 
\begin{equation}\label{stripairn}
q = r = \frac{2(n+2)}n,\ \ \ \ \ \ \ q' = r' = \frac{2(n+2)}{n+4},
\end{equation}
i.e., the exponents in the original Strichartz inequality \eqref{stri}.

\vskip 0.2in

Theorems A and B below are distinguished by the following different assumptions on the behavior of $V(t)$ defined by \eqref{V}.

\vskip 0.2in

\noindent \textbf{Hypothesis (A).} There exists $\gamma>0$ such that for every $t> 0$ the following inequality holds:
\begin{equation}\label{H0}
V(t) \ge \gamma\ t^{D} e^{t \tr B}.
\end{equation}

\vskip 0.2in

\noindent \textbf{Hypothesis (B).} There exists $\gamma>0$ and a number $2\le D_\infty<D$ such that for every $t> 0$ the following inequality holds:
\begin{equation}\label{HB}
V(t) \ge \gamma\ \min\{t^{D}, t^{D_\infty}\}.
\end{equation}

\vskip 0.2in

The relevant findings can be summarized as follows:
\begin{itemize}
\item[(a)] Theorem A applies when Hypothesis (A) holds.
\item[(b)] Theorem B applies when Hypothesis (B) holds.
\end{itemize}

\vskip 0.2in

\noindent We are finally ready to state our main a priori estimates for solutions of the Cauchy problem \eqref{A}. 

\vskip 0.2in

\noindent \textbf{Theorem A.}\label{T:strichartzone}\ 
\emph{Assume \eqref{H0} in} Hypothesis (A). \emph{Suppose that $r>2$ and that the pair $(q,r)$ be admissible for \eqref{A}, with $r < \frac{2D}{D-2}$ when $D> 2$. There exists $C = C(n,r)>0$ such that if $u$ solves \eqref{A}, with $\vf\in L^2(\Rn)$ and $F$ such that $||e^{\frac{\tr B}{2} t} ||F(\cdot,t)||_{L^{r'}_x}||_{L^{q'}_t}<\infty$, then the following Strichartz estimate holds}
\begin{align}\label{strichartzone}
& \bigg(\int_\R e^{\frac{q\tr B}2 t} \big(\int_{\Rn}|u(x,t)|^r dx\big)^{\frac qr} dt\bigg)^{\frac 1q} \le C \bigg\{\|\vf\|_{L^2(\Rn)}
 + \bigg(\int_\R e^{\frac{q' \tr B}2 t} \big(\int_{\Rn}|F(x,t)|^{r'} dx\big)^{\frac{q'}{r'}} dt\bigg)^{\frac 1{q'}}\bigg\}.
\end{align}

\vskip 0.2in

\noindent We emphasize that the estimate \eqref{strichartzone} holds true for the Strichartz pairs relative to $D$, i.e., when $q=r$ are as in \eqref{rr}. If, additionally, $\tr B = 0$, we obtain 
\begin{equation}\label{strichartze}
||u||_{L^{\frac{2(D+2)}D}(\R^{n+1})} \le C \left(||\vf||_{L^2(\Rn)} + ||F||_{L^{\frac{2(D+2)}{D+4}}(\R^{n+1})}\right).
\end{equation}
This estimate applies, in particular, to the case when the matrix $B$ in \eqref{A} takes the form \eqref{barB}. In this case, the non-isotropic dilations \eqref{sdl} imply that \eqref{strichartze} is best possible.

It is worth noting that, even in the non-degenerate case when $Q>0$, in which $D=n$, the estimate \eqref{strichartze} generalizes \eqref{stri}, since it allows for a drift, see Example \ref{E:inv}, or the case $k=n$ of Example \ref{E:fan}. Similarly, the more general inequality \eqref{strichartzone} extends the well-known mixed-space Strichartz estimates  in \cite{GVstrich}, \cite{Caze}. However, the full scope of Theorem A goes well beyond the case of an invertible $Q$. For a more comprehensive discussion of such scope, we refer the reader to Sections \ref{S:truth}, \ref{S:volume} and \ref{S:example}. We next introduce our second main result.

\vskip 0.2in

\noindent \textbf{Theorem B.}\label{T:strichartzoneKra}\ 
\emph{Suppose that \eqref{HB} in} Hypothesis (B) \emph{hold. Assume that $r>2$, and that $r<\frac{2D}{D-2}$ when $D>2$. Suppose further that $(q,r)$ satisfy \eqref{admissible}, and that $1<q_{\infty} <\infty$ is such that
\begin{align}\label{inftyadm}
\frac{2}{q_{\infty}} \le D_\infty(\frac{1}{2}-\frac{1}{r}).
\end{align} 
Then there exists $C = C(n,r)>0$ such that for every $\vf\in L^2(\R^{n})$ and $F$ such that \newline $||e^{\frac{\tr B}{2} t} ||F(\cdot,t)||_{L^{r'}_x}||_{L^{q'}_t\cap L^{q_{\infty}'}_{t}}<\infty$, one has}
\begin{align}\label{strichartzoneK0}
\|e^{\frac{\tr B}2 t} ||u(\cdot,t)||_{L^r(\Rn)}\|_{L^{q}_{t}+L^{q_{\infty}}_t} \le C \bigg\{&  \ \|\vf\|_{L^2(\Rn)}
+ ||e^{\frac{\tr B}{2} t} ||F(\cdot,t)||_{L^{r'}_x}||_{L^{q'}_t\cap L^{q_{\infty}'}_{t}}\bigg\}.
\end{align}

\vskip 0.2in

\subsection{Validity of Hypothesis (A) or (B)}\label{S:truth} 
Having presented our main results, we next show that Hypothesis (A) or (B) is generically satisfied in \emph{all} possible scenarios for the matrices $Q$ and $B$. Our first result in this direction is the following.

\begin{theorem}\label{T:trB}
Suppose that $Q$ and $B$ satisfy \emph{one} of the following hypothesis \emph{(i)-(iii)}:
\begin{itemize}
\item[(i)] $\sigma(B)\not\subset i\R$;
\item[(ii)] $\sigma\{B\}\subset i \R$ and $\operatorname{Rank}(Q) = n$; 
\item[(iii)] $B$ is nilpotent and of the form 
\begin{equation}\label{barB}
\overline B = \begin{pmatrix} 0 & 0 & \cdot & \cdot & \cdot & 0 & 0
\\
B_1 & 0 & \cdot & \cdot & \cdot & 0 & 0
\\
0 & B_2 & \cdot & \cdot & \cdot &  0 & 0
\\
\cdot & \cdot & \cdot & \cdot & \cdot & \cdot & \cdot
\\
\cdot & \cdot & \cdot & \cdot & \cdot & \cdot & \cdot
\\
\cdot & \cdot & \cdot & \cdot & \cdot & \cdot & \cdot
\\
0 & 0 & \cdot & \cdot & \cdot &  B_r & 0
\end{pmatrix}.
\end{equation}
\end{itemize}
Then \eqref{H0} in \emph{Hypothesis (A)} holds.
\end{theorem}

\vskip 0.2in

\noindent Concerning the hypothesis (iii) in Theorem \ref{T:trB}, we emphasise that, when we write  $\overline B$ in the form \eqref{barB}, we intend that such matrix has been obtained from $B$ in the canonical form \eqref{QB}, by replacing all the blocks designated by a $\star$ with a zero matrix of the same dimensions. The matrix $\overline B$ has a special interest since, by the results in \cite{Ku,Ku2}, \cite{LP}, we infer that, when $F = 0$ in \eqref{A}, the resulting  Schr\"odinger equation
\begin{equation}\label{AA}
\p_t u = i \operatorname{tr}(Q \nabla^2 u) +  \langle B x,\nabla u\rangle 
\end{equation}
 is homogeneous of degree two with respect to a family of non-isotropic dilations \emph{if and only if} the matrix of the drift is in the form \eqref{barB}. In such case, the relevant dilations are constructed as follows. The ranks $p_0, p_1,...,p_r$ of the matrices $Q_0, B_1,...,B_r$ in \eqref{QB} induce a corresponding stratification of the ambient space 
\[
\Rn=\R^{p_0}\oplus\R^{p_1}\oplus\cdots\oplus\R^{p_r}.
\]
We accordingly denote by $x=\left(x^{(p_0)},x^{(p_1)},\ldots,x^{(p_r)}\right)$ a generic point of $\Rn$, and define the non-isotropic dilations  
\begin{equation}\label{sdl}
\delta_\lambda(x,t) = \left(\lambda
x^{(p_0)},\lambda^3
x^{(p_1)},\ldots,\lambda^{2r+1}x^{(p_r)},\lambda^2t\right).
\end{equation}
Then the free Schr\"odinger equation 
\begin{equation}\label{AAdil}
\p_t u = i \operatorname{tr}(Q \nabla^2 u) +  \langle \overline B x,\nabla u\rangle 
\end{equation}
is invariant under the scalings \eqref{sdl}. We stress that, in such case, the number $D$ in \eqref{hd} is precisely the \emph{homogeneous dimension}, in the sense of \cite{FS}, \cite{RS}, \cite{NSW}, of the spatial component of the dilations \eqref{sdl}. We also note that, when $Q$ is invertible, the matrix $\overline B$ must necessarily be the zero matrix. 
In Section \ref{S:volume} we will say more about the volume function $\overline V(t)$ associated with \eqref{AAdil}.

\vskip 0.2in

The assumption (ii) in Theorem \ref{T:trB} does not encompass the case in which $\sigma(B) \subset i \R$, but $\operatorname{Rank}(Q)<n$ (hence, $Q$ is not invertible). Under such circumstances, the estimate \eqref{H0} can fail significantly, as the Example \ref{E:imspec} shows. This example serves as a paradigm for the situation covered by the next result. 

\begin{theorem}\label{T:B}
Suppose that $\operatorname{Rank}(Q)<n$, and that $B$ is similar to a skew-symmetric matrix. Then \eqref{HB} in \emph{Hypothesis (B)} holds with $D_\infty = n$.
\end{theorem}

To prove Theorem \ref{T:B} we establish a delicate precise asymptotic estimate of the volume function which is of independent interest. Precisely, we show that for some $\gamma>0$, one has for $t\to \infty$
\begin{align}\label{T:Bexp0}
V(t)= \gamma t^{n}+O(t^{n-1}).
\end{align}
We observe that \eqref{T:Bexp0} trivially implies the existence of  $C, t_0>0$ such that:
\[
V(2t)\le C V(t),\ \ \ \ \ \  \text{for every}\ t\ge t_0.
\]
We also note that the hypothesis on the matrix $B$ in Theorem \ref{T:B} implies, in particular, that $\sigma(B)\subset i \R$. The reader should also observe that under the assumptions (ii) or (iii) in Theorem \ref{T:trB}, or the hypothesis on $B$ in Theorem \ref{T:B}, we must necessarily have $\tr B = 0$.

\vskip 0.2in

\noindent \underline{The anomalous case:} It is important to note that Theorems \ref{T:trB} and  \ref{T:B} cover nearly all possible scenarios. However, there is one final possibility which is left out, when the matrices $Q$ and $B$ satisfy \emph{all} three of the following assumptions: 
\begin{itemize}
\item[(iv)] $\operatorname{Rank}(Q)<n$; 
\item[(v)] $\sigma(B)\subset i \R$, and $B$\ is not similar to a skew-symmetric matrix; 
\item[(vi)] $B$ is not in the form \eqref{barB}.
\end{itemize}
Since (v) implies that $\tr B=0$, by (i) of Proposition \ref{P:boom0}, we know that there exists $ D_{\infty}\ge 2$ and $\gamma>0$ such that for any $t\ge1$ the volume function satisfies
\begin{align}\label{an}
V(t)\ge \gamma t^{D_{\infty}}.
\end{align}
But under the hypothesis (iv)-(vi)
the function $V(t)$ exhibits a behaviour that defies a single classification. Specifically, in the Example \ref{E:B} we prove that both the possibilities
$D_\infty \ge D$ and $D_\infty < D$ can occur. The situation can be thus summarized:
\begin{itemize}
\item[(A)] If $D_\infty \ge D$, then \eqref{an} implies that \eqref{H0} holds, and Theorem A applies. 
\item[(B)] If instead $D_\infty <D$, then from \eqref{an} we infer that \eqref{HB} holds, and thus Theorem B applies.
\end{itemize}

\medskip

\vskip 0.2in

\subsection{Description of the paper} 
A brief outline of the present work is as follows. In Section \ref{S:mot}, we discuss the motivation and background for our main results. Specifically, we show that a change to conformal self-similar coordinates in the Schr\"odinger equation \eqref{CPs} leads to a problem of the form \eqref{A}. We also explain the emergence of exponential weights in Theorems A and B by showing how they arise naturally through a change to such coordinates in the classical estimates \eqref{ginvel} of Ginibre and Velo.
In Section \ref{S:ext}, starting from the representation \eqref{erf}, we extend the semigroup defined by \eqref{er} to the entire real line and establish Theorem \ref{T:de}. Section \ref{S:T} is dedicated to the proof of Theorem A, drawing on the well-established $T^\star T$ argument introduced by Ginibre and Velo in \cite{GV, GVstrich}, whereas Section \ref{S:TB} is devoted to the proof of Theorem B. Section \ref{S:volume} focuses on the large-time behavior of the volume function 
$V(t)$ defined in \eqref{V}, with the primary goal of proving Theorems \ref{T:trB} and \ref{T:B}. 
The concluding Section \ref{S:example} presents several examples that illustrate the results. Of particular interest is Proposition \ref{E:fan}, as it demonstrates the variety of situations that can arise. Section \ref{S:prelim} collects background results that will be used in the rest of the paper.

As a final note, we mention that this work does not address sharp endpoint estimates, in the spirit of Keel and Tao \cite{KT}, nor other fundamental issues such as short-time existence and regularity, as discussed in \cite{GV, GVstrich}, \cite{Caze}, and \cite{Tao}. These topics, along with potential applications of our results to nonlinear Schr\"odinger equations, will be the focus of a future study.

%%%%%%%%%%%%%%%%%%%%%%%%%%%%%%%%%%%%%%%%%%

%%%%%%%%%%%%%%%%%%%%%%%%%%%%%%%%%%%%%%%%%%%%%%%%%%%%%%%%

%%%%%%%%%%%%%%%%%%%%%%%%%%%%%%%%%%%%%%%%%%%%%%%%%%%%%%%%%%%%

\vskip 0.2in

\subsection{Notation} 
We use the standard notation $\omega_n = \frac{2\pi^{n/2}}{\G(n/2+1)}$ for the $n$-dimensional measure of the unit ball in $\Rn$. The notation $i\R$ indicates the imaginary axis in the complex plane $\mathbb C$. We denote by $C^\star = [c_{ji}]$ the transpose of a matrix $C=[c_{ij}]$. 
The notation $\operatorname{tr} C$ indicates the trace of $C$, $\sigma(C)$ its spectrum.  When for a matrix we write $C_{h\times k}$, we want to emphasise that its dimensions are $h\times k$. We use such notation for partitioned matrices. The notation $\nabla^2 u$ indicates the Hessian matrix of a function $u:\Rn\times \R\to \R$ with respect to the variable $x\in\Rn$. Given a measurable function $u:\Rn\times \R\to \overline \R$, we denote by $||u||_{L^q(\R,L^r(\Rn))} = \left(\int_{\R} ||u(\cdot,t)||^q_{L^r(\Rn)} dt\right)^{1/q}$. We indicate by $L^{q_{1}} + L^{q_{2}}$ the
Banach space of functions $f$ such that $f = f_{1} + f_{2}$ with $f_{1}\in L^{q_{1}}$ and
$f_{2} \in L^{q_{2}}$, endowed with the norm
\begin{align*}
\|f\|_{L^{q_{1}}+L^{q_{2}}}=\inf_{f=f_{1}+f_{2},f_{1}\in L^{q_{1}},f_{2} \in L^{q_{2}}}\|f_{1}\|_{L^{q_{1}}}+\|f_{2}\|_{L^{q_{2}}}.
\end{align*}
We also need the Banach space $L^{p}\cap L^{q}$, endowed with the norm
\begin{align*}
\|f\|_{L^{p}\cap L^{q}}=\|f\|_{L^{p}}+\|f\|_{L^{q}}.
\end{align*}
We recall that $(L^{p}+L^{q})'=L^{p'}\cap L^{q'}$,
see \cite[Sec. 3]{BS}. The Fourier transform employed in this paper is 
\[
\mathscr F f(\xi) = \hat f(\xi) = \int_{\Rn} e^{-2\pi i\sa \xi,x\da} f(x) dx.
\]

%%%%%%%%%%%%%%%%%%%%%%%%%%%%%%%%%%%%%%%%%%%%%%%%%%%%%%%%%%%%%%%

\medskip

\subsection{Acknowledgement.} We thank Dennis Bernstein for kindly providing us with \eqref{Qtilde} and \eqref{trace2}. This led to extending our original two-dimensional version of Lemma \ref{L:ber} to arbitrary dimension. We also thank Carlos Kenig for kindly bringing to our attention the cited works of A. Ionescu.  

\vskip 0.2in
%%%%%%%%%%%%%%%%%%%%%%%%%%%%%%%%%%%%%%%%%%%%%%%%%%%

\section{Motivation and context}\label{S:mot}

We begin by discussing a scenario particularly relevant to the study of blow-up formation and dynamics in semilinear Schr\"odinger equations. The following considerations illustrate one possible way in which problems such as \eqref{A} arise
and how the drift contributes to the emergence of exponential weights in the relevant Strichartz estimates.

\subsection{Conformal coordinates and exponential weights}\label{S:exp}
Consider a function $u:\Rn\times [0,T)\to \R$. If we make the change of dependent variable
\begin{equation}\label{zeroo}
u(x,t) = v(y,\tau),\ \ \ \ \text{where}\ \ \ \ (y,\tau) = \Phi(x,t) = (\alpha(t) x,\beta(t)),
\end{equation}
then a computation gives
\begin{equation*}
\p_\tau v + \frac{\alpha'(t)}{\alpha(t)\beta'(t)} \sa y,\nabla v\da - i \frac{\alpha(t)^2}{\beta'(t)} \Delta_y v = \frac{1}{\beta'(t)} \left[\p_t u - i \Delta_x u\right].
\end{equation*}
Imposing the conditions 
\begin{equation}\label{bprimo}
\alpha'(t) = \alpha(t)^3,\ \ \ \ \ \ \beta'(t) = \alpha(t)^2,
\end{equation}
we further obtain 
\begin{equation}\label{duee}
\p_\tau v - i  \Delta_y v  + \sa y,\nabla v\da = \frac{1}{\alpha(t)^2} \left[\p_t u - i \Delta_x u\right].
\end{equation}
Integrating the equations in \eqref{bprimo}, we find
\begin{equation}\label{beta}
\alpha(t) = \frac{1}{\sqrt{2(T-t)}},\ \ \ \ \ \beta(t) = \frac 12 \log \frac{1}{T-t},\ \ \ \ 0\le t<T.
\end{equation}
The transformation \eqref{zeroo}, \eqref{beta} represent a change to \emph{conformal self-similar coordinates}. These coordinates first appeared in the blow-up formation for the heat equation in the papers \cite{We}, \cite{GK}. For the Schr\"odinger equation, they have been extensively used in seminal works such as \cite{MR,MR2,MR3,MRS2,MRS}, see also \cite{Su}.  By \eqref{duee} it is clear that, if $u(x,t)$ solves the problem \eqref{CPs} in $\Rn\times [0,T)$, 
then with 
\[
\tilde F(y,\tau) = 2(T-t) F(x,t),\ \ \ \ \ \ \ \tau_0 = \log \frac 1{\sqrt T},
\]
the function $v(y,\tau)$, defined as in \eqref{zeroo}, solves the Cauchy problem of type \eqref{A} in $\Rn\times [\tau_0,\infty)$ 
\begin{equation}\label{quattroo}
\p_\tau v  - i  \Delta v + \sa y,\nabla v\da= \tilde F(y,\tau),\ \ \ \ \ v(y,\tau_0) = \vf(2\sqrt T y),
\end{equation}
where $Q = - B = I_n$\footnote{We note that the parabolic counterpart of \eqref{quattroo} is the Ornstein-Uhlenbeck operator in \cite{OU}}. 
Note that, although there exist no dilations associated with \eqref{quattroo}, by the invertibility of  $Q$, and by (b) following Definition \ref{D:hd}, we know that the local homogeneous dimension of \eqref{quattroo} is $D = n$. Consequently, by Definition \ref{D:admissible} we know that  for $r>2$, with $r < \frac{2n}{n-2}$ when $n> 2$, the pair $(q,r)$ is \emph{admissible} when
\begin{equation}\label{aqr}
\frac 2q = n(\frac 12 - \frac 1r).
\end{equation}
Note that these are the same as the admissible pairs of Ginibre and Velo for \eqref{CPs} (see \cite{GVstrich} and \cite{Caze}).
Furthermore, since $\sigma(B) = \{-1\}$, we are in the hypothesis (i) of Theorem \ref{T:trB}. Therefore, \eqref{H0} holds and Theorem A applies, giving for \eqref{quattroo}
\begin{align}\label{ou}
& \bigg(\int_{\tilde I} e^{\frac{q\tr B}2 t} \big(\int_{\Rn}|u(x,t)|^r dx\big)^{\frac qr} dt\bigg)^{\frac 1q} \le C \bigg\{\|\vf\|_{L^2(\Rn)}
 + \bigg(\int_{\tilde I} e^{\frac{q' \tr B}2 t} \big(\int_{\Rn}|F(x,t)|^{r'} dx\big)^{\frac{q'}{r'}} dt\bigg)^{\frac 1{q'}}\bigg\},
\end{align}
where $\tilde I = [\tau_0,\infty)$.
It is important to keep in mind that, since $B = - I_n$, we have $\tr B = - n$ in \eqref{ou}.

Confronted with \eqref{ou}, the reader may wonder about the appearance of exponential weights. We want to show next that their presence is \emph{natural}. With this in mind, let us go back to the well-known mixed-norm estimates for \eqref{CPs} 
\begin{equation}\label{ginvel}
\bigg(\int_I \big(\int_{\Rn}|u(x,t)|^r dx\big)^{\frac qr} dt\bigg)^{\frac 1q} \le C \bigg\{\|\vf\|_{L^2(\Rn)}
 + \bigg(\int_I \big(\int_{\Rn}|F(x,t)|^{r'} dx\big)^{\frac{q'}{r'}} dt\bigg)^{\frac 1{q'}}\bigg\},
\end{equation}
where $I = (0,T)$, and the pair $(q,r)$ satisfies \eqref{aqr}. The transformations \eqref{zeroo}, \eqref{beta} give
\begin{align*}
& \bigg(\int_I \big(\int_{\Rn}|u(x,t)|^r dx\big)^{\frac qr} dt\bigg)^{\frac 1q} = \bigg(\int_I \big(\int_{\Rn}|v(\frac{x}{\sqrt{2(T-t)}},\frac 12 \log \frac{1}{T-t})|^r dx\big)^{\frac qr} dt\bigg)^{\frac 1q}
\\
&  = \bigg(\int_I \big(\int_{\Rn}|v(y,\frac 12 \log \frac{1}{T-t})|^r (2(T-t))^{\frac n2} dy\big)^{\frac qr} dt\bigg)^{\frac 1q}.
\end{align*}
We now further make the change of variable $\tau = \frac 12 \log \frac{1}{T-t}$, so that $T-t = e^{-2\tau}$ and $dt = 2 e^{-2\tau} d\tau$. This gives
\begin{align*}
& \bigg(\int_I \big(\int_{\Rn}|u(x,t)|^r dx\big)^{\frac qr} dt\bigg)^{\frac 1q} = C \bigg(\int_{\tilde I} e^{-2\tau(\frac{nq}{2r}+1)} \big(\int_{\Rn}|v(y,\tau)|^r dy\big)^{\frac qr} d\tau\bigg)^{\frac 1q}.
\end{align*}
Keeping \eqref{aqr} in mind, we conclude
\begin{align*}
& \bigg(\int_I \big(\int_{\Rn}|u(x,t)|^r dx\big)^{\frac qr} dt\bigg)^{\frac 1q} = C \bigg(\int_{\tilde I} e^{-\frac{nq}{2}\tau} \big(\int_{\Rn}|v(y,\tau)|^r dy\big)^{\frac qr} d\tau\bigg)^{\frac 1q}.
\end{align*}
Similar considerations show that
\begin{align*}
& \bigg(\int_I \big(\int_{\Rn}|F(x,t)|^{r'} dx\big)^{\frac{q'}{r'}} dt\bigg)^{\frac 1{q'}} = C 
  \bigg(\int_{\tilde I} e^{-\frac{nq'}{2}\tau } \big(\int_{\Rn}|\tilde F(y,\tau)|^{r'} dy\big)^{\frac{q'}{r'}} d\tau\bigg)^{\frac 1{q'}}.
\end{align*}
Since, as we have observed above, $\tr B = - n$, 
we see that  
\[
e^{-\frac{nq}{2}\tau }  = e^{\frac{q\tr B}2 \tau},\ \ \ \ \ e^{-\frac{nq'}{2}\tau } = e^{\frac{q'\tr B}2 \tau}. 
\]
We have thus re-obtained \eqref{ou} from the  classical  Strichartz estimates for the Schr\"odinger equation in \eqref{CPs}.

These computations explain the appearance of the exponential weights in our results. They also show that, via the transformation \eqref{zeroo}, \eqref{beta}, our general Strichartz estimate \eqref{strichartzone} in Theorem A, when specialized to the operator in \eqref{quattroo}, is nothing but  the Ginibre-Velo mixed-norm inequality \eqref{ginvel} in disguise.
 
As a final comment, we add that the presence of the exponential weights is tightly connected to the Lie group invariance in \eqref{Lie}. This is an aspect that echoes a similar feature in the works of A. Ionescu on semisimple Lie groups of real rank one, see \cite{Iomrl, Iojfa, Iojfa2, Ioann,  Ioduke}. We thank C. Kenig for bringing this to our attention.

\subsection{Twisted Laplacians} Another interesting aspect of \eqref{A} is connected with the Cauchy problem for the Schr\"odinger equation on Lie groups of Heisenberg type or, more in general, on a Carnot group $\bG$ of step two. For important contributions we refer the reader to \cite{Mu, BG, BBG, BFG, BF}. Letting $\bg = \mathfrak h \oplus \mathfrak v$ indicate the Lie algebra of such a group, with inner product $\langle \cdot,\cdot\rangle$, and logarithmic coordinates $x = (z,\sigma)$, consider the Kaplan mapping $J: \mathfrak v \to \operatorname{End}(\mathfrak h)$, defined by 
\begin{equation}\label{kap}
\langle J(\sigma)z,z'\rangle = \langle [z,z'],\sigma\rangle = - \langle J(\sigma)z',z\rangle,
\end{equation}
see \cite{Ka}. Clearly, $J(\sigma)^\star = - J(\sigma)$, and one has $\langle J(\sigma)z,z\rangle = 0$. 
With respect to a given orthonormal basis $\{e_1,...,e_m\}$ of $\mathfrak h$, the sub-Laplacian on $\bG$ can be expressed by the formula
\begin{equation}\label{sl}
\mathscr L = \Delta_z + \frac 14 \sum_{\ell,\ell' = 1}^k \langle J(\ve_\ell)z,J(\ve_{\ell'})z\rangle \p_{\sigma_\ell}\p_{\sigma_{\ell'}} + \sum_{\ell = 1}^k  \p_{\sigma_\ell} \Theta_\ell,
\end{equation}
where $\Delta_z$ represents the standard Laplacian in the variable $z = (z_1,...,z_m)$, and for a fixed orthonormal basis $\{\ve_1,...,\ve_k\}$ of $\mathfrak v$, we have let \begin{equation}\label{thetaell0}
\Theta_\ell = \sum_{s=1}^m \langle J(\ve_\ell)z,e_s\rangle \p_{z_s},\ \ \ \ \ \ell = 1,...,k.
\end{equation}
When $\bG$ is of Heisenberg type, one has $J(\sigma)^2 = - |\sigma|^2 I_{\mathfrak h}$, and \eqref{sl} becomes
\[
\mathscr L = \Delta_z + \frac{|z|^2}4 \Delta_\sigma  + \sum_{\ell = 1}^k \p_{\sigma_\ell} \Theta_\ell,
\]
see \cite{Ka} and \cite{CDKR}. 
Consider now the Cauchy problem for the Schr\"odinger equation for \eqref{sl} in $\bG\times (0,\infty)$,
\begin{equation}\label{cp}
\p_t u - i \mathscr L u,\ \ \ \ \ \ \ u(x,0) = \vf(x),\ \  \ \ x\in \bG,
\end{equation}
and take Fourier transform in the central variable $\sigma\in \mathfrak v$, 
\[
\hat u(z,\la,t) = \int_{\mathfrak v} e^{-2\pi i\langle\la,\sigma\rangle} u(z,\sigma,t) d\sigma.
\]
If for a fixed $\la\in \mathfrak v$, we let $\tilde v(z,t) = \hat u(z,\la,t)$, we obtain from \eqref{cp} that such function must satisfy the problem
\begin{equation}\label{cp2}
\p_t \tilde v  = i\left(\Delta_z \tilde v  - \pi^2 |J(\la)z|^2\ \tilde v\right) - 2 \pi  \langle J(\la) z, \nabla_z \tilde v\rangle,
\ \ \ \ \ 
\tilde v(z,0) = \hat \vf(z,\la) \ \ \  z\in \mathfrak h,\ t>0.
\end{equation}
When $\bG$ is the Heisenberg group, or more in general a group of Heisenberg type, the PDE in \eqref{cp2} is the Schr\"odinger equation for the \emph{twisted Laplacian} (for this operator, see \cite{Fo}, \cite{Strjfa}, \cite{Veluma}).
Comparing with \eqref{A}, we see that, for any fixed $\la\in \mathfrak v$, we can represent the PDE in \eqref{cp2} in the form
\begin{equation}\label{cp3}
\p_t \tilde v  = i\left(\Delta_z \tilde v  - \frac{1}4|Bz|^2\ \tilde v\right) +\sa B z, \nabla_z \tilde v\rangle,
\end{equation}
where $B =   
- 2 \pi J(\la)$.  Note that $B^\star = - B$. Skew-symmetric drifts will be thoroughly analyzed in the present work.

\vskip 0.2in

\subsection{Background} It is well-known that the Gramian matrix \eqref{Ds} 
has the property that $Q(t_0) >0$ for some $t_0>0$ if and only if  $Q(t)>0$ for all $t>0$. This follows immediately from the identity 
\begin{equation}\label{mono}
Q(t+s) = Q(t) + e^{tB} Q(s) e^{tB^\star},\ \ \ \ \ \ \ \ s, t >0.
\end{equation}

\noindent This was first recognised by Kalman, who  proved in \cite[Cor. 5.5]{Kal} that either \textbf{(H)} or the invertibility of $Q(t)$ are necessary and sufficient for the complete controllability of the linear system in $\Rn$
\[
y'(t) = B y(t) + A u(t),\ \ \ \ \ \ y(0) = x,\ \ \ \ \ A A^\star = Q.
\] 

\noindent  In PDEs the matrix \eqref{Ds} first arose in connection with the equation 
\begin{equation}\label{AAA}
\p_t u -  \tr(Q\nabla^2 u) -  \sa Bx,\nabla u\da   = 0.
\end{equation}
This equation, along with its various generalizations, has inspired extensive research over the past four decades across several fields, including partial differential equations (especially kinetic PDEs), stochastic differential equations, probability, semigroups, and control theory.
The classical heat equation is obtained from taking $Q = I_n$ and $B = O_n$. The choice $Q = I_n = - B$ gives the Ornstein-Uhlenbeck equation \cite{OU}. Under the hypothesis \textbf{(H)}, the equation \eqref{AAA} admits a notable positive fundamental solution 
\begin{equation}\label{p}
p(x,y,t) = \frac{(4\pi)^{-\frac n2}}{\sqrt{V(t)}} \exp\left( - \frac{\sa Q(t)^{-1}(e^{tB}x-y),e^{tB}x-y\da}{4}\right).
\end{equation}
A special case of the kernel $p(x,y,t)$ first appeared in Kolmogorov's note \cite{Kol}. In the form given in \eqref{p}, it was introduced in Dym's work \cite[p.134]{Dym} on the flow of stochastic differential equations driven by white noise, where the author examined the highly degenerate case when $Q$ has rank one, and $B$ is the companion matrix (see \cite[Sec.5.2]{Bernie} for this notion) associated with the monic polynomial $p(s) = s^n - a_1 s^{n-1}-...-a_{n-1} s - a_n$. 
For general $Q$ and $B$, the 
formula \eqref{p} was derived in the introduction of the fundamental work \cite[p.148]{Ho}. Extending Kolmogorov's result, H\"ormander showed that \textbf{(H)} implies (and is, in fact,  equivalent to) the hypoellipticity of \eqref{AAA}. A partial list of both older and more recent papers on \eqref{AAA} that are pertinent to the present work  includes \cite{OU}, \cite{Kol},  \cite{Bri}, \cite{Ho}, \cite{Ku, Ku2},  \cite{LP}, \cite{Fre}, \cite{GTma} \cite{BGT}.

%%%%%%%%%%%%%%%%%%%%%%%%%%%%%%
%%%%%%%%%%%%%%%%%%%%%%%%%%%%%%%%%%%%%%%%%%%%%%%%%%%%%%%%%%%%%%%%%

\section{The Schr\"odinger group}\label{S:ext}

We note that the differential operator on the left-hand side of the equation in \eqref{A}  is  invariant under the non-commutative group law: 
\begin{equation}\label{Lie}
(x,s)\circ (y,t) = (y+ e^{-tB}x,s+t),
\end{equation}
where $x, y\in \Rn$ and $s, t\in \R$. Equipped with \eqref{Lie}, the space $(\R^{n+1},\circ)$ forms a non-Abelian Lie group whose identity element is the origin. Such Lie group invariance suggests to transform the solution $u$ in \eqref{A} (when $F = 0$), according to the equation
\begin{equation}\label{drift}
v(x,t) = u(e^{- t B} x,t).
\end{equation}
The function $v$ in \eqref{drift} is easily seen to satisfy the Cauchy problem without drift 
\begin{equation}\label{equiv}
\p_t v - i \operatorname{tr}(Q'(t) \nabla^2 v) = 0,\ \ \ \ \ v(x,0) = \vf(x),
\end{equation}
where we have denoted by $Q'(t) = e^{tB} Q e^{tB^\star} \ge 0$ the derivative of the matrix defined by \eqref{Ds}.
After taking Fourier transform with respect to $x$ in \eqref{equiv}, one is  led to
\begin{equation}\label{hatv}
\hat v(\xi,t) = \hat \vf(\xi) e^{-4\pi^2 i \sa Q(t) \xi,\xi\da}.
\end{equation}
Since, as noted, \textbf{(H)} is equivalent to $Q(t)>0$ for every $t>0$, it is possible to invert the Fourier transform in \eqref{hatv}, and obtain the following solution formula for \eqref{A} (when $F = 0$) 
\begin{equation}\label{er}
\mathcal T(t) \vf(x)  = 
 \frac{(4\pi
)^{-\frac{n}{2}}e^{-\frac{i \pi n}4}}{\sqrt{V(t)}}   \int_{\Rn} e^{i \frac{\sa Q(t)^{-1}(e^{tB}x-y),e^{tB}x - y\da}{4}} \vf(y) dy,\ \ \ \ t>0, 
\end{equation}
see \cite[Prop.2.4]{GL}. The operator $\mathcal T(t)$ defines a semigroup on $\mathscr S(\Rn)$ which can be uniquely extended to a strongly continuous one on $L^2(\Rn)$. Furthermore, for any $r\ge 2$ and $t>0$ one can extend \eqref{er} to a  bounded operator $\mathcal T(t) : L^{r'}(\Rn)\to L^{r}(\Rn)$ such that for any $\vf\in L^{r'}(\Rn)$ the following dispersive inequality holds (see \cite[Theor. 4.1]{GL})
\begin{equation}\label{dis}
||\mathcal T(t) \vf||_{L^{r}(\Rn)} \le C(n,r)  \frac{e^{-\frac{\tr B}{r} t}}{V(t)^{\frac 12 - \frac 1{r}}}\ ||\vf||_{L^{r'}(\Rn)}.
\end{equation} 
When $Q = I_n$ and $B = O_n$, the matrix in \eqref{Ds} simplifies to $Q(t) = t I_n$, and thus $\det Q(t) = t^n$. In this case, formula \eqref{er} gives the classical kernel of the  Schr\"odinger equation  in \eqref{CPs}, yielding no new insights. However, as shown in Section \ref{S:exp}, when $B\not= O_n$, formula \eqref{er} becomes an interesting object of study, even when $Q$ is invertible.

The purpose of this section is to extend the semigroup defined by \eqref{er} to a group on the whole real line. 
For the next result we refer the reader to e.g. \cite[Lemma 4.1]{HMMS}.

\begin{lemma}\label{LemmaExpansion}
Let $U(t)$ be a $C_{0}$-semigroup and $t_{0}>0$ be such that $U(t_{0})$ is invertible. Then $U(t)$ is invertible for all $t\ge 0$ and it can be extended to a $C_{0}$-group $(U(t))_{t\in \mathbb{R}}$ by setting
\[
U(t)=U(-t)^{-1}, \ \ \ \ \ \ t<0.
\]
\end{lemma}

\begin{theorem}\label{ExtendingProp}
The semigroup $\{\mathcal T(t)\}_{t\ge 0}$ on $\mathscr S(\Rn)$ defined by \eqref{er} is extended  to a group $\{U(t)\}_{t\in \R}$ by the following formula
\begin{align}\label{ExtendedSemigroup}
U(t)\phi(x)=\begin{cases}
(4\pi)^{-\frac{n}{2}}\frac{e^{i\frac{\pi n}{4}} e^{|t| \tr B}}{\sqrt{V(|t|)}}\int_{\Rn}e^{-i\frac{\langle Q(|t|)^{-1}(x-e^{|t|B}y),x-e^{|t|B} y\rangle}{4}}\phi(y)dy \ \ \ \text{if }t<0,
\\
\phi(x) \ \ \ \ \text{if }t=0,
\\
(4\pi)^{-\frac{n}{2}}\frac{e^{-i\frac{\pi n}{4}}}{\sqrt{V(t)}}\int_{\Rn}e^{i\frac{\langle Q(t)^{-1}(y-e^{tB}x),y-e^{tB}x\rangle}{4}}\phi(y)dy \ \ \ \ \ \ \ \ \ \ \text{if }t>0.
\end{cases}
\end{align}
\end{theorem}

\begin{proof}
In order to apply Lemma \ref{LemmaExpansion} we need to show that for some $t>0$ the formula \eqref{er} defines an invertible operator on $\mathscr S(\Rn)$. For this it is useful to work on the Fourier transform side. From \cite[(2.15)]{GL} we know that for every $t>0$ one has
\begin{equation}\label{eccola}
\widehat{\mathcal T(t)\vf}(\xi) = e^{-t \tr B} \hat \vf(e^{-tB^\star}\xi) \  e^{-4\pi^2 i  \sa Q(t) e^{-tB^\star}\xi,e^{-tB^\star}\xi\da} \overset{def}{=} \hat{g}(\xi).
\end{equation}    
We obtain from \eqref{eccola}
\begin{align}\label{opop}
&\hat{\phi}(\xi)=e^{t \tr B}e^{4\pi^{2}i\langle Q(t)\xi,\xi\rangle}\hat{g}(e^{t B^{\star}}\xi)\implies\phi =\mathscr{F}^{-1}(e^{t \tr B}e^{4\pi^{2}i\langle Q(t)\cdot,\cdot\rangle}\hat{g}(e^{t B^{\star}}\cdot)).
\end{align}
Recall now that for $A\in G\ell(\R,n)$
\[
\hat{g}(A^{\star}\xi)=|\det A|^{-1}\widehat{g \circ A^{-1}}(\xi),
\]
and consequently we have
\begin{align*}
\hat{g}(e^{t B^{\star}}\xi)=e^{-t \tr B}\mathscr F(g\circ e^{-t B})(\xi).
\end{align*}
Inserting the latter two formulas in \eqref{opop}, we thus find
\begin{equation}\label{fi}
\vf = \mathscr{F}^{-1}\left(e^{4\pi^{2}i\langle Q(t)\cdot,\cdot\rangle}\right) \star (g\circ e^{-t B}).
\end{equation}
We next recall that, given $A\in G\ell(\mathbb C,n)$ such that $A^\star = A$ and $\Re A \ge 0$, one has 
\begin{equation}\label{gengaussi2}
\mathscr F\left(\frac{1}{\sqrt{\operatorname{det} A}}(4\pi
)^{-\frac{n}{2}} e^{- \frac{\sa A^{-1}\cdot,\cdot\da}{4}}\right)(\xi) =
e^{- 4 \pi^2  \sa A\xi,\xi\da},
\end{equation}
where $\sqrt{\operatorname{det} A}$ is the unique analytic branch such that $\sqrt{\operatorname{det} A}>0$ when $A$ is real, see \cite[Theor.7.6.1]{Hobook}.
Taking $A=- i Q(t)$, we infer
\[
\mathscr{F}^{-1}\left(e^{4\pi^{2}i\langle Q(t)\cdot,\cdot\rangle}\right) = \frac{(4\pi
)^{-\frac{n}{2}} e^{i\frac{\pi n}{4}}}{\sqrt{V(t)}} e^{- i\frac{\sa Q(t)^{-1}\cdot,\cdot\da}{4}} 
\]
We then have from \eqref{fi}
\begin{align}\label{Tallamenuno}
\vf(x) = & \mathcal T(t)^{-1} g(x) = \frac{(4\pi
)^{-\frac{n}{2}} e^{i\frac{\pi n}{4}}}{\sqrt{V(t)}} \int_{\Rn} e^{- i\frac{\sa Q(t)^{-1}(x-y),(x-y)\da}{4}} g(e^{-t B}y)dy.
\end{align}
If we now define as in Lemma \ref{LemmaExpansion}
\[
U(t) = \begin{cases}
\mathcal T(t)\ \ \ \ \ \ \ \ \ \ \ t\ge 0,
\\
\mathcal T(-t)^{-1}\ \ \ \ \ \ t< 0,
\end{cases}
\]
then from \eqref{Tallamenuno} we reach the desired conclusion \eqref{ExtendedSemigroup}.

\end{proof}

In what follows we will need the following consequence of \eqref{eccola} and \eqref{opop}. 
\begin{corollary}\label{C:eccopop}
Let $U(t)$ be the group given by \eqref{ExtendedSemigroup}, we have
\begin{align}\label{FTsemigrou}
\widehat{U(t)\phi}(\xi)=\begin{cases}
e^{|t| \tr B}e^{4\pi^{2}i\langle Q(|t|)\xi,\xi\rangle}\hat{\phi}(e^{|t|B^{\star}}\xi) \ \ \ \ \ \ \ \ \ \ \ \ \ \ \ \ \text{if }t<0,
\\
\hat{\phi}(\xi) \ \ \ \text{if }t=0
\\
e^{-t \tr B}e^{-4\pi^{2}i\langle Q(t)e^{-tB^{\star}}\xi, e^{-tB^{\star}}\xi\rangle}\hat{\phi}(e^{-tB^{\star}}\xi) \ \ \ \text{if }t>0.
\end{cases}
\end{align}
\end{corollary}

It will be convenient to have the following corollary of \eqref{ExtendedSemigroup} whose proof we leave to the reader.

\begin{proposition}\label{P:ft}
Suppose that $\vf\in \So$. Then 
\[
U(t) \vf(x) = 
\begin{cases}
(4\pi
)^{-\frac{n}{2}}  \frac{e^{\frac{i \pi n}4} e^{|t| \tr B}}{\sqrt{V(|t|)}} e^{-i \frac{|Q(|t|)^{-1/2} x|^2}{4}}\mathscr F\left(\vf e^{-i \frac{|Q(|t|)^{-1/2}e^{|t|B} \cdot|^2}{4}}\right)(-(4\pi)^{-1} Q(|t|)^{-1} e^{|t|B^\star} x),\ \ t<0,
\\
(4\pi
)^{-\frac{n}{2}}  \frac{e^{-\frac{i \pi n}4}}{\sqrt{V(t)}} e^{i \frac{|Q(t)^{-1/2} e^{tB} x|^2}{4}}\mathscr F\left(\vf e^{i \frac{|Q(t)^{-1/2} \cdot|^2}{4}}\right)((4\pi)^{-1} Q(t)^{-1} e^{tB} x),\ \ t>0.
\end{cases}
\]
\end{proposition}

Using Proposition \ref{P:ft}, and proceeding as in the proof of \cite[Theor.4.1]{GL}, we can now extend to negative times the dispersive estimate in \eqref{dis}.

\begin{theorem}\label{T:de}
Let $r\ge 2$. For any $\vf\in L^{r'}(\Rn)$ one has
\[
||U(t) \vf||_{L^{r}(\Rn)} \le 
\begin{cases}
C(n,r) \ \frac{e^{\frac{\tr B}{r'} |t|}}{V(|t|)^{\frac 12 - \frac 1{r}}}\ ||\vf||_{L^{r'}(\Rn)},\ \ t<0,
\\
C(n,r) \ \frac{e^{-\frac{\tr B}{r} t}}{V(t)^{\frac 12 - \frac 1{r}}}\ ||\vf||_{L^{r'}(\Rn)},\ \ t>0.
\end{cases}
\]
\end{theorem}
The reader should note the two distinct exponentials on the right-hand side of the inequalities in Theorem \ref{T:de}. At first glance, their discrepancy may seem concerning, but in reality, a small miracle occurs during the proof of Theorem A, and everything will eventually fall into place.

\vskip 0.2in

%%%%%%%%%%%%%%%%%%%%%%%%%%%%%%%%%%%%%%%%%%%%%%%%%%%%%%%%%%
%%%%%%%%%%%%%%%%%%%%%%%%%%%%%%%%
%%%%%%%%%%%%%%%%%%%%%%%

%%%%%%%%%%%%%%%%%%%%%%%%%%%%%%%%%%%%%%%%%%%%%%%%%%%%%%%%%%%%%%%%%%%%%%%%%%%%%%%%%%%%%%%%%

\medskip

%%%%%%%%%%%%%%%%%%%%%%%%%%%%%%%%%%%%%%%%%%%%%%%%%%%%%%%%%%%%%%%%%%%%%%%%%%%%%%%%%%%%%%%%%%%%%%%%%%%%%%%%%%%%

%%%%%%%%%%%%%%%%%%%%%%%%%%%%%%%%%%%%%%%%%%%%%%%%%%%%%%%

%%%%%%%%%%%%%%%%%%%%%%%%%%%%%%%%%%%%%%%%%%%%%%%%%%%%%%%%%

\section{Proof of Theorem A}\label{S:T}

In this section we suitably adapt to the Cauchy problem \eqref{A} the ideas of Ginibre and Velo in \cite{GV, GVstrich}. Our final objective is proving Theorem A. We begin by observing that by the Duhamel principle, given $\vf\in \So$ and $F\in \mathscr S(\R^{n+1})$, the solution to \eqref{A} is given by 
\begin{equation}\label{duhamel}
u(x,t) = U(t)\vf(x) + \int_0^t U(t-s)(F(\cdot,s))(x) ds.
\end{equation}
We begin by introducing an operator $T^\star :\So\longrightarrow \mathscr C^\infty(\R^{n+1})$ as follows
\begin{equation}\label{Tstar}
T^\star(\vf)(x,t) = e^{\frac{t \tr B}2} U(t)\vf(x),
\end{equation}
where $\{U(t)\}_{t\in \R}$ is the group defined by \eqref{ExtendedSemigroup}.

\begin{proposition}[Unitarity]\label{P:Tstar}
We have
\[
T^\star: L^2(\Rn)\ \longrightarrow\ L^{\infty,2}(\R^{n+1}),
\]
and moreover for every $\vf\in L^2(\Rn)$ one has
\begin{equation}\label{Tstaruno}
||T^\star(\vf)||_{L^\infty_t L^2_x} = ||\vf||_{L^2_x}.
\end{equation}
\end{proposition}

\begin{proof}
We have from \eqref{FTsemigrou} for every $t>0$
\begin{align*}
& ||T^\star(\vf)(\cdot,t)||_{L^2_x} = ||\mathscr F(T^\star(\vf))(\cdot,t)||_{L^2_x} = e^{\frac{t \tr B}2} ||\widehat{U(t)\vf}||_{L^2_x}
\\
& = e^{\frac{t \tr B}2} e^{- t \tr B} \left(\int_{\Rn} |\hat{\vf}(e^{-tB^{\star}}\xi)|^2 d\xi\right)^{1/2} 
\\
& = e^{\frac{t \tr B}2} e^{- t \tr B} e^{\frac{t \tr B}2} ||\vf||_{L^2_x} = ||\vf||_{L^2_x}.
\end{align*}
If instead $t<0$, again from \eqref{FTsemigrou} we have 
\begin{align*}
& ||T^\star(\vf)(\cdot,t)||_{L^2_x} = e^{-\frac{|t| \tr B}2} e^{|t| \tr B} \left(\int_{\Rn} |\hat{\phi}(e^{|t|B^{\star}}\xi)|^2 d\xi\right)^{1/2} = ||\vf||_{L^2_x}.
\end{align*}
These identities prove \eqref{Tstaruno}.

\end{proof}

We next compute the operator $T$, whose adjoint is $T^\star$. Denoting by $\sa \cdot,\cdot \da$ the inner product in $L^2(\Rn)$, and by $\sa\sa\cdot,\cdot\da\da$ that in $L^2(\R^{n+1})$, for every $\vf\in \So$ and $F\in \mathscr S(\R^{n+1})$, we  have from \eqref{Tstar} 
{\allowdisplaybreaks
\begin{align*}
& \sa T(F),\vf\da = \sa\sa F,T^\star(\vf)\da\da = \int_\R \int_{\Rn} F(x,t) \overline{T^\star(\vf)(x,t)}\ dx dt
\\
& =  \int_\R e^{\frac{t \tr B}2} \int_{\Rn} F(x,t)  \overline{U(t)\vf(x)}\  dx dt = \int_\R e^{\frac{t \tr B}2} \int_{\Rn} \hat F(\xi,t)  \overline{\widehat{U(t)\vf}(\xi)}\  d\xi dt 
\\
& = \int_{-\infty}^0 e^{\frac{t \tr B}2} \int_{\Rn} \hat F(\xi,t)  \overline{\widehat{U(t)\vf}(\xi)}\  d\xi dt  + \int_0^\infty e^{\frac{t \tr B}2} \int_{\Rn} \hat F(\xi,t)  \overline{\widehat{U(t)\vf}(\xi)}\  d\xi dt 
\\
&  = \int_{-\infty}^0 e^{-\frac{|t| \tr B}2} \int_{\Rn} \hat F(\xi,t) e^{|t| \tr B} e^{-4\pi^{2}i\langle Q(|t|)\xi,\xi\rangle} \overline{\hat{\phi}(e^{|t|B^{\star}}\xi)}\  d\xi dt 
\\
& + \int_0^\infty e^{\frac{t \tr B}2} \int_{\Rn} \hat F(\xi,t) e^{-t \tr B}e^{4\pi^{2}i\langle Q(t)e^{-tB^{\star}}\xi, e^{-tB^{\star}}\xi\rangle} \overline{\hat{\phi}(e^{-tB^{\star}}\xi)}\  d\xi dt 
\\
&  = \int_{-\infty}^0 e^{\frac{|t| \tr B}2} \int_{\Rn} e^{-|t| \tr B} \hat F(e^{-|t|B^{\star}}\eta,t) e^{-4\pi^{2}i \langle Q(|t|) e^{-|t|B^\star} \eta,e^{-|t|B^\star \eta}\rangle} \overline{\hat{\phi}(\eta)}\  d\eta dt 
\\
& + \int_0^\infty e^{-\frac{t \tr B}2} \int_{\Rn} e^{t \tr B}\hat F(e^{tB^\star}\eta,t) e^{4\pi^{2}i\langle Q(t)\eta,\eta\rangle} \overline{\hat{\phi}(\eta)}\  d\eta dt 
\\
& = \int_{-\infty}^0 e^{\frac{|t| \tr B}2} \int_{\Rn}  \mathscr F\left(U(|t|)(F(\cdot,t)\right)(\eta)\ \overline{\hat{\phi}(\eta)}\  d\eta dt 
\\
& + \int_0^\infty e^{-\frac{t \tr B}2} \int_{\Rn} \mathscr F\left(U(|t|)(F(\cdot,t)\right)(\eta)\ \overline{\hat{\phi}(\eta)}\  d\eta dt  
\\
& = \int_{-\infty}^0 e^{\frac{|t| \tr B}2} \int_{\Rn} U(-t)\left(F(\cdot,t)\right)(y)\ \overline{\phi(y)}\  dy dt 
\\
& +  \int_0^\infty e^{-\frac{t \tr B}2} \int_{\Rn} U(-t)\left(F(\cdot,t)\right)(y)\ \overline{\phi(y)}\  dy dt 
\end{align*}}
This chain of equalities, and the definition \eqref{Tstar}, show that
\begin{equation}\label{T}
T(F)(y) = \int_\R e^{-\frac{t \tr B}2} U(-t)(F(\cdot,t))(y) dt = \int_\R T^\star(F(\cdot,t))(y,-t) dt.
\end{equation}
It is immediate from \eqref{T} and \eqref{Tstaruno} that
\[
||T(F)||_{L^2_y} \le \int_\R ||T^\star(F(\cdot,t))(\cdot,-t)||_{L^2_y} dt = \int_\R ||F(\cdot,-t)||_{L^2_y} dt = ||F||_{L^1_t L^2_y}.
\]
We summarize these results in the following.

\begin{proposition}\label{P:T}
The operator defined by \eqref{T} can be uniquely extended to a bounded linear operator 
\[
T: L^{1,2}(\R^{n+1})\ \longrightarrow\ L^2(\Rn),
\]
satisfying the estimate
\begin{equation}\label{Tuno}
||T(F)||_{L^2_y} \le ||F||_{L^1_t L^2_y},
\end{equation}
for every $F\in L^{1,2}(\R^{n+1})$.
\end{proposition}

Our next step is to prove the following critical result.

\begin{theorem}[First Strichartz estimate]\label{T:SEA} 
Assume \eqref{H0} in \emph{Hypothesis (A)}. Suppose that $r>2$, and that $r < \frac{2D}{D-2}$ when $D> 2$. If $(q,r)$ is admissible for \eqref{A}, then there exists $C(n,r,q)>0$ such that for every $F\in \mathscr S(\R^{n+1})$, one has  
\begin{equation}\label{gettingthereA}
||T^\star T(F)||_{L^q_tL^r_x} \le C(n,r,q)\ ||F||_{L^{q'}_tL^{r'}_x}.
\end{equation}
\end{theorem}

\begin{proof}

From \eqref{Tstar} and \eqref{T} we have for $F\in \mathscr S(\R^{n+1})$,
\begin{align*}
& T^\star(T(F))(x,t) = e^{\frac{t \tr B}2} U(t)(T(F))(x) = e^{\frac{t \tr B}2} U(t) \int_\R e^{-\frac{s \tr B}2} U(-s)(F(\cdot,s))(x) ds
\\
& =  \int_\R e^{\frac{(t-s) \tr B}2} U(t-s)(F(\cdot,s))(x) ds
\\
& = \int_{-\infty}^t e^{\frac{(t-s) \tr B}2} U(t-s)(F(\cdot,s))(x) ds + \int_t^\infty e^{\frac{(t-s) \tr B}2} U(t-s)(F(\cdot,s))(x) ds
\\
& = I(x,t) + II(x,t).
\end{align*}
Since in the integral in $I(x,t)$ we have $t-s\ge 0$, from the first part of Theorem \ref{T:de} we find for any $r\ge 2$
\begin{align*}
& ||I(\cdot,t)||_{L^r_x} \le \int_{-\infty}^t e^{\frac{(t-s) \tr B}2} ||U(t-s)(F(\cdot,s))||_{L^r_x} ds
\\
& \le C(n,r)\int_{-\infty}^t \frac{e^{(t-s) \tr B\left(\frac 12-\frac 1r\right)}}{V(t-s)^{\frac 12 - \frac 1{r}}}\ ||F(\cdot,s)||_{L^{r'}_x} ds. 
\end{align*}
Since in the integral in $II(x,t)$ we have $t-s\le 0$, from the second part of Theorem \ref{T:de} we have
\begin{align*}
& ||II(\cdot,t)||_{L^r_x} \le \int_t^\infty e^{\frac{(t-s) \tr B}2} ||U(t-s)(F(\cdot,s))||_{L^r_x} ds
\\
& \le C(n,r) \int_t^\infty \frac{e^{(s-t) \tr B\left(\frac 12-\frac 1r\right)}}{V(s-t)^{\frac 12 - \frac 1{r}}}\ ||F(\cdot,s)||_{L^{r'}_x} ds. 
\end{align*}
Combining the latter three equations, we conclude
\begin{align}\label{nice}
&  ||T^\star(T(F))(\cdot,t)||_{L^r_x} \le C(n,r) \int_\R \left\{\frac{e^{|t-s| \tr B}}{V(|t-s|)}\right\}^{\left(\frac 12-\frac 1r\right)}\ ||F(\cdot,s)||_{L^{r'}_x} ds. 
\end{align}

In view of the assumption \eqref{H0}, we obtain from \eqref{nice} the following conclusion
\begin{equation}\label{supernice}
||T^\star(T(F))(\cdot,t)||_{L^r_x} \le C(n,\gamma,r) \int_\R \frac{||F(\cdot,s)||_{L^{r'}_x}}{|t-s|^{D\left(\frac 12-\frac 1r\right)}}\  ds. 
\end{equation}
If we now introduce the parameter $\beta$ by the equation
\begin{equation}\label{etabetaD}
1-\beta = D(\frac 12-\frac{1}{r}),
\end{equation}
and consider the Riesz operator of fractional integration on the line
\[
\mathcal I_\beta(h)(t) = \int_{\R} \frac{h(s)}{|t-s|^{1-\beta}} ds,
\]
then it is clear from \eqref{supernice} that 
\begin{equation}\label{good}
||T^\star(T(F))(\cdot,t)||_{L^r_x} \le C(n,\gamma,r)\ \mathcal I_\beta(||F(\cdot,\cdot)||_{L^{r'}_x})(t).
\end{equation}
Note that the bound $\beta\le 1$ is automatically true since $r\ge 2$. Furthermore, $\beta = 1$ if and only of $r=2$, so when $r>2$ we have $\beta < 1$ for any $D\ge 2$. The inequality $\beta >0$ instead, forces the condition 
\begin{equation}\label{derange}
\frac 12-\frac{1}{D}  < \frac 1r.
\end{equation}
When $D=2$ this is automatically true for any $2\le r < \infty$, whereas when $D>2$ it imposes that $2\le r < \frac{2D}{D-2}$. Since Definition \ref{D:admissible} includes the restrictions $2<r<\infty$, and $r< \frac{2D}{D-2}$ when $D>2$, we are guaranteed that the number $\beta$ in \eqref{etabetaD} satisfy $0<\beta<1$. If we thus take $1<p<\frac{1}\beta$, and $\frac 1p - \frac 1q = \beta$, then the Hardy-Littlewood-Sobolev theorem implies that
\begin{equation}\label{HLSA}
\mathcal I_\beta : L^p(\R)\to L^q(\R)\ \Longleftrightarrow\ \frac 1p - \frac 1q = \beta = 1 - D(\frac 12-\frac{1}{r}),
\end{equation}
and moreover
\begin{equation}\label{marcellinoA}
||\mathcal I_\beta(h)||_{L^q(\R)} \le C(p,\beta) ||h||_{L^p(\R)},
\end{equation}
see \cite[Theor. 1, p. 119]{St}. 
By the hypothesis that $(q,r)$ be admissible for \eqref{A}, we are allowed to take $p = q'<q$ in the equation $\frac 1p - \frac 1q = \beta$. This choice gives in fact
\[
\frac 2q = 1 - \beta = D(\frac 12-\frac{1}{r}),
\]
which is precisely \eqref{admissible}. From \eqref{good} and \eqref{marcellinoA}, we conclude the validity of \eqref{gettingthereA}.

\end{proof}

Next, we use Theorem \ref{T:SEA} to establish the following result.
 
\begin{theorem}[Second Strichartz estimate]\label{T:SE2A}
Under the hypothesis of Theorem \ref{T:SEA}, for every $\vf\in L^2(\Rn)$ one has 
\begin{align}\label{SE2A}
\|T^{\star}(\vf)\|_{L^q_tL^r_x}\le C(n,r,q) \|\vf\|_{L^2_x}.
\end{align}
\end{theorem}

\begin{proof}
First, one uses Theorem \ref{T:SEA} to obtain the following basic estimate for the operator $T$. Observe that  for every $F\in \mathscr S(\R^{n+1})$ we have 
\begin{align}\label{TF}
& \|T(F)\|_{L^2_x}^{2}=|\langle T(F),T(F) \rangle|=|\langle T^{\star}T(F),F\rangle |
\\
& \le \|T^{\star}T(F)\|_{L^q_t L^r_x}\|F\|_{L^{q'}_t L^{r'}_x}\le C(n,r,q)\ \|F\|^{2}_{L^{q'}_t L^{r'}_x}.	
\notag
\end{align}
But then we have for $\vf\in L^2(\Rn)$
\begin{align*}
|\langle T^{\star}(\vf),F\rangle |=|\langle \vf, T(F) \rangle |\le \|\vf\|_{L^2_x} \|T(F)\|_{L^2_x}\le C(n,r,q) \|\vf\|_{L^2_x} \|F\|_{L^{q'}_t L^{r'}_x}.
\end{align*}
By taking the supremum of all $F\in L^{q'}_{t}L^{r'}_{x}\cap L^{1}_{t}L^{2}_{x}$, we conclude that \eqref{SE2A} holds.

\end{proof}

With these results in hand, we can finally give the 

\begin{proof}[Proof of Theorem A]
Recalling \eqref{Tstar}, we now reformulate \eqref{SE2A} in the following way:
\begin{equation}\label{wow}
\left(\int_\R e^{\frac{q\tr B}2 t}  \left(\int_{\Rn}|U(t)\vf(x)|^r dx\right)^{\frac qr} dt\right)^{\frac 1q} \le C(n,r,q) \|\vf\|_{L^2_x}.
\end{equation}
This estimate establishes the homogeneous part of \eqref{strichartzone}. We now turn to analysing the non-homogeneous term in \eqref{duhamel}.
From the first estimate in Theorem \ref{T:de} we have
\begin{align*}
& ||\int_0^t U(t-s)(F(\cdot,s))(\cdot) ds||_{L^r_x}\le \int_0^t ||U(t-s)(F(\cdot,s))(\cdot)||_{L^r_x} ds
\\
& \le C(n,r) \int_0^t \frac{e^{-\frac{\tr B}{r}(t-s)}}{V(t-s)^{\frac 12 - \frac 1{r}}}\ ||F(\cdot,s)||_{L^{r'}_x} ds
\\
& \le C(n,r,\gamma) \int_0^t \frac{e^{-\frac{\tr B}{r}(t-s)} e^{-\left(\frac 12 - \frac 1{r}\right) \tr B (t-s)}}{(t-s)^{D\left(\frac 12 - \frac 1{r}\right)}}\ ||F(\cdot,s)||_{L^{r'}_x} ds,
\end{align*}
where in the last inequality we have used the assumption \eqref{H0}. This estimate implies 
\begin{align*}
& e^{\frac{\tr B}{2} t}\ ||\int_0^t U(t-s)(F(\cdot,s))(\cdot) ds||_{L^r_x} \le C(n,r,\gamma) \int_0^t \frac{e^{\frac{\tr B}{2} s} ||F(\cdot,s)||_{L^{r'}_x}}{(t-s)^{D\left(\frac 12 - \frac 1{r}\right)}} ds
\\
& = C(n,r,\gamma) \int_\R \frac{h(s)}{(t-s)^{D\left(\frac 12 - \frac 1{r}\right)}} ds = C(n,r,\gamma)\ \mathcal I_\beta(h)(t),
\end{align*}
where, with $\beta$ as in \eqref{etabetaD}, we have let
\[
h(s) = e^{\frac{\tr B}{2} s} ||F(\cdot,s)||_{L^{r'}_x}.
\]
Using again \eqref{marcellinoA}, we infer
\begin{align}\label{nonhomU}
& \left(\int_\R e^{t \frac{q\tr B}2} ||\int_0^t U(t-s)(F(\cdot,s))(\cdot) ds||^q_{L^r_x} dt\right)^{\frac 1q}
 \le C^\star(n,r,\gamma,D)\ ||e^{\frac{\tr B}{2} t} ||F(\cdot,t)||_{L^{r'}_x}||_{L^{q'}_t}.
\end{align}
Combining \eqref{wow} and \eqref{nonhomU}, and keeping the definition of \eqref{Tstar} in mind, we have finally proved \eqref{strichartzone}, thus completing the proof of Theorem \ref{T:strichartzone}.

\end{proof}

\subsection{Ginibre-Velo implies restriction}\label{S:res} We close this section with a remark concerning the connection between the  estimate \eqref{TF} and the restriction theorem for the Fourier transform. Consider the classical Schr\"odinger equation, for which $Q = I_n$ and $B = O_n$. In such case, if we use Plancherel in the inequality \eqref{TF}, we obtain 
\begin{align}\label{TFplan}
& \|\widehat{T(F)}\|_{L^2(\Rn)} \le C(n,r)\ \|F\|_{L^{q'}_t L^{r'}_x}.	
\end{align}
Keeping in mind that we presently have $D=n$,  if we apply \eqref{TFplan} with $q' = r'$ as in \eqref{stripairn}, we obtain
\begin{equation}\label{res}
\left(\int_{\Rn} |\widehat{T(F)}(\xi)|^2 d\xi\right)^{1/2} \le C(n)\ \|F\|_{L^{\frac{2(n+2)}{n+4}}(\R^{n+1})}.
\end{equation}
The inequality \eqref{res} easily implies the Tomas-Stein restriction inequality for the truncated paraboloid. To see this, it suffices to observe that, when $Q = I_n$ and $B = O_n$, the operator \eqref{T} becomes
\[
T(F)(y) = \int_\R U(-t)(F(\cdot,t))(y) dt,
\]
where $U(t) = e^{it\Delta}$. Therefore, 
if we take the Fourier transform of \eqref{T}, we obtain
\begin{equation}\label{res1}
\widehat{T(F)}(\xi) = \int_\R e^{4\pi^2 i t |\xi|^2} \hat F(\xi,t) dt =  \int_{\R^{n+1}} e^{-2\pi i\sa(\xi,-2\pi |\xi|^2),(x,t)\da} F(x,t) dx dt = \hat F(\xi,-2\pi|\xi|^2).
\end{equation}
If for any $M>0$ we consider the quadratic hypersurface
\[
S_M = \{(\xi,\tau)\in\R^{n+1}\mid \tau = - 2\pi |\xi|^2,\ |\xi|\le M\},
\]
then the Riemannian measure on $S$ is given by
\[
d\sigma = \sqrt{1+16\pi^2 |\xi|^2}.
\]
Combining \eqref{res} with \eqref{res1}, we conclude that 
there exists $C(n,M)>0$ such that for any $F\in \mathscr S(\R^{n+1})$
\begin{equation}\label{TS}
\left(\int_{S_M} |\hat F|^2 d\sigma\right)^{1/2} \le C(n,M)\ \|F\|_{L^{\frac{2(n+2)}{n+4}}(\R^{n+1})}.
\end{equation}
Thus, the approach of Ginibre and Velo provides an elementary proof of the restriction theorem \eqref{TS} for the Fourier transform.

%%%%%%%%%%%%%%%%%%%%%%%%%%%%%%%%%%%%%%%%%%%%%%%%%%%%%%%%%%%%%%

\section{Proof of Theorem B}\label{S:TB}

The aim of this section is to establish Strichartz estimates for \eqref{A} in situations, such as Example \ref{E:imspec}, when the hypothesis \eqref{H0} fails, and we must use \eqref{HB} instead. Our objective is proving Theorem B. 
Recall that a measurable function $g$ belongs to the weak $L^{r}$ space $L^{r,w}$ if
\begin{align*}
\|g\|_{L^{r,w}}^{r} =\sup_{\la>0}\la^{r}\mu (\{t\mid |g(t)|>\la\})<\infty,
\end{align*}
where we have denoted by $\mu$ the measure on the relevant space.
Let $1<p,q,r<\infty$ satisfy $\frac{1}{p}+\frac{1}{r}=\frac{1}{q}+1$. Then there exists a constant $C$ such that for any $f\in L^{p}$ and $g\in L^{r,w}$ one has
\begin{align}\label{weakYoung}
\|f\star g\|_{q}\le C \|f\|_{p}\|g\|_{r,w}.
\end{align}
For the proof of \eqref{weakYoung}, see \cite{O}. The next lemma is well-known, but for completeness we recall its simple proof. 

\begin{lemma}\label{WighFraInt}
Let $0<\gamma_{1}, \gamma_2 <1$, $C_1,C_2>0$. Let $k:\mathbb{R}\to \mathbb{R}$ be such that
\begin{align*}
|k(t)|\le \begin{cases}
\frac{C_1}{|t|^{\gamma_{1}}} \ \ \ \ \text{if }|t|\le 1,\\
\frac{C_2}{|t|^{\gamma_{2}}} \ \ \ \ \text{if }|t|\ge1.
\end{cases}
\end{align*}
If $1<p_{1}<q_{1}<\infty$ , $1< p_{2}, q_{2} <\infty$
\begin{align*}
\gamma_{1}=1+\frac{1}{q_{1}}-\frac{1}{p_{1}}, \ \ \ \text{and}\ \ \ \ \gamma_{2}\ge1+\frac{1}{q_{2}}-\frac{1}{p_{2}},
\end{align*}  
then one has
\begin{align}\label{meglio}
\|f\star k\|_{L^{q_{1}}+L^{q_{2}}}\le C \|f\|_{L^{p_{1}}\cap L^{p_{2}}}.
\end{align}
\end{lemma}

\begin{proof}
Let $f\in L^{p_{1}}\cap L^{p_{2}}$, and set $I=(-1,1)$. If we let $k_1 = \mathbf 1_I\ k$ and $k_2 = \mathbf 1_{\R \setminus I}\ k$, we can write 
\begin{align*}
f\star k= f\star k_1 + f\star k_2.
\end{align*}
Since $\frac{1}{p_{1}}-\frac{1}{q_{1}}=1-\gamma_{1}$, the Hardy-Littlewood-Sobolev inequality (or the fact that $k_1\in L^{1/\gamma_1,w}(\R)$ and \eqref{weakYoung}) implies
\begin{align*}
\|k_{1}\star f\|_{L^{q_{1}}}\lesssim \|f\|_{L^{p_{1}}}.
\end{align*}
For $\la>0$ we now have
\[
\{t\in \R\mid |k_2(t)|>\la\} \subset \{|t|\ge 1\mid \frac{C_2}{|t|^{\gamma_2}} > \la\} = \{|t|\ge 1\mid |t|< (C_2\la^{-1})^{1/\gamma_2}\},
\]
provided that $0<\la \le C_2$. If instead $\la >C_2$, then $|\{t\in \R\mid |k_2(t)|>\la\}| = 0$. For any $r>1$ we thus have
\begin{align}\label{distr}
& \sup_{0<\la}\la^{r} |\{t\in \R\mid |k_2(t)|>\la\}| = \sup_{0<\la<C_2}\la^{r} |\{t\in \R\mid |k_2(t)|>\la\}|
\\
& \le \sup_{0<\la<C_2} \la^{r}(C_2\la^{-1})^{1/\gamma_2} \le \overline C,
\notag
\end{align}
provided that $r\gamma_{2}\ge1$, or equivalently $1 - \gamma_2 \le 1 - \frac 1r$. If we take 
\[
\frac{1}{p_{2}}-\frac{1}{q_{2}}=1-\frac{1}{r} \ge 1 - \gamma_2,
\] 
we conclude from \eqref{weakYoung} and \eqref{distr} that 
\[
||k_2 \star f||_{L^2} \le ||k_2||_{L^{r,\infty}} ||f||_{L^{p_2}} \le \overline C ||f||_{L^{p_2}}.
\]
This proves \eqref{meglio}. 

\end{proof}

The next two results represent the counterpart of Theorems \ref{T:SEA} and \ref{T:SE2A} in situations in which \eqref{HB} in the Hypothesis (B) hold. We only provide details of the former. 

\begin{theorem}\label{T:SE2K} 
Assume \eqref{HB}, and that $r>2$ and $r<\frac{2D}{D-2}$ when $D>2$.  
Suppose further that $(q,r)$ satisfy \eqref{admissible}, and that $1<q_{\infty} <\infty$ is such that
\begin{align}\label{inftyadm2}
D_{\infty}(\frac{1}{2}-\frac{1}{r})\ge \frac{2}{q_{\infty}}.
\end{align} 
Then there exists $C = C(n,r,\gamma, D, D_\infty)>0$ 
such that for every $F\in \mathscr S(\R^{n+1})$, one has  
\begin{equation}\label{FStriDrift}
\|T^{\star}T(F)\|_{L^{q}_{t}+L^{q_{\infty}}_{t} L^{r}_{x}}\le C \|F\|_{L^{q'}_{t}\cap L^{q'_{\infty}}_{t} L^{r'}_{x}}.
	\end{equation}
\end{theorem}

\begin{proof}
If in the estimate \eqref{nice} we use the assumption \eqref{HB} instead of \eqref{H0}, we obtain
\begin{align*}
\|T^{\star}T(F)\|_{L^{q_{1}}_{t}+L^{q_{2}}_{t}L^{r}_{x}}&\lesssim \|\int_{\R}\frac{\|F(s,\cdot)\|_{L^{r'}_{x}(\mathbb{R}^{m})}}{\min\{|t-s|^{D(\frac{1}{2}-\frac{1}{r})},|t-s|^{D_{\infty}(\frac{1}{2}-\frac{1}{r})}\}}ds\|_{L^{q_{1}}_{t}+L^{q_{2}}_{t}}. 
\end{align*}
By our choice of $r$, we know that $0<D(\frac{1}{2}-\frac{1}{r})<1$, see \eqref{derange} and the discussion following it. Therefore, if
\begin{align}
1<p_{1}<q_{1}<\infty, \ \ D(\frac{1}{2}-\frac{1}{r})=1+\frac{1}{q_{1}}-\frac{1}{p_{1}},
\end{align}   
and  
\begin{align*}
D_{\infty}(\frac{1}{2}-\frac{1}{r})\ge1+\frac{1}{q_{2}}-\frac{1}{p_{2}},
\end{align*}
then applying Lemma \ref{WighFraInt} we infer that
\begin{align*}
\|T^{\star}T(F)\|_{L^{q_{1}}_{t}+L^{q_{2}}_{t}L^{r}_{x}}\lesssim \|F\|_{L^{p_{1}}_{t}\cap L^{p_{2}}_{t} L^{r'}_{x}}.
\end{align*}
To complete the proof of \eqref{FStriDrift}  it suffices to take $q_1 = q, p_1 = q'$,  $q_{2}=q_{\infty}$, $p_2 = q'_{2} = q_\infty'$. This gives
\begin{align*}
D(\frac{1}{2}-\frac{1}{r})=\frac{2}{q},\ \ \ D_{\infty}(\frac{1}{2}-\frac{1}{r})\ge\frac{2}{q_{\infty}},
\end{align*}
which are precisely the standing assumptions.

\end{proof}

\begin{theorem}\label{T:SE2K2}
Assume \eqref{HB} and suppose that $r$,$q$ and $q_{\infty}$ be as in Theorem \ref{T:SE2K}. Then for every $\vf\in L^2(\Rn)$ one has 
\begin{align}\label{SE22K}
		\|T^{\star}\vf\|_{L^{q}_{t}+L^{q_{\infty}}_{t} L^{r}_{x}}\le C(n,r,q) \|\vf\|_{L^2(\Rn)}.
	\end{align}
\end{theorem}

\begin{proof}
See the proof of Theorem \ref{T:SE2A}.
\end{proof}

\medskip

We can now provide the

\begin{proof}[Proof of Theorem \ref{T:strichartzoneKra}]
From the Strichartz estimate \eqref{SE22K} we know that
\begin{equation}\label{woww}
\|e^{\frac{\tr B}2 t}  \left(\int_{\Rn}|U(t)\vf(x)|^r dx\right)^{\frac{1}{r}}\|_{L^{q}_{t}+L^{q_{\infty}}_{t}} \le C(r,q) \|\vf\|_{L^2_x}.
\end{equation}
This estimate establishes the homogeneous part of \eqref{duhamel}. To control the non-homogeneous term, as in the proof of Theorem \ref{T:SE2K} we have
\begin{align*}
& ||\int_0^t U(t-s)(F(\cdot,s))(\cdot) ds||_{L^r_x}\le \int_0^t ||U(t-s)(F(\cdot,s))(\cdot)||_{L^r_x} ds
\\
& \le C(n,r) \int_0^t \frac{e^{-\frac{\tr B}{r}(t-s)}}{V(t-s)^{\frac 12 - \frac 1{r}}}\ ||F(\cdot,s)||_{L^{r'}_x} ds
\\
& \le C(n,r,\gamma) \int_0^t \frac{e^{-\frac{\tr B}{r}(t-s)} e^{-\left(\frac 12 - \frac 1{r}\right) \tr B (t-s)}}{\min\{(t-s)^{D\left(\frac 12 - \frac 1{r}\right)},(t-s)^{D_{\infty}\left(\frac 12 - \frac 1{r}\right)}\}}\ ||F(\cdot,s)||_{L^{r'}_x} ds.
\end{align*}
This estimate implies 
\begin{align*}
& e^{\frac{\tr B}{2} t}\ ||\int_0^t U(t-s)(F(\cdot,s))(\cdot) ds||_{L^r_x} \le C(n,r,\gamma) \int_0^t \frac{e^{\frac{\tr B}{2} s} ||F(\cdot,s)||_{L^{r'}_x}}{\min\{(t-s)^{D\left(\frac 12 - \frac 1{r}\right)},(t-s)^{D_{\infty}\left(\frac 12 - \frac 1{r}\right)}\}} ds
\\
& = C(n,r,\gamma) \int_\R \frac{h(s)}{\min\{(t-s)^{D\left(\frac 12 - \frac 1{r}\right)},(t-s)^{D_{\infty}\left(\frac 12 - \frac 1{r}\right)}\}} ds
\end{align*}
where $
h(s) = e^{\frac{\tr B}{2} s} ||F(\cdot,s)||_{L^{r'}_x}.$
Again by Lemma \ref{WighFraInt}, we finally obtain
\begin{align*}
& \|\int_\R e^{t \frac{\tr B}2} ||\int_0^t U(t-s)(F(\cdot,s))(\cdot) ds||_{L^r_x} dt\|_{L^{q}_{t}+L^{q_{\infty}}}
\le C^\star(n,r,\gamma)\ ||e^{\frac{\tr B}{2} t} ||F(\cdot,t)||_{L^{r'}_x}||_{L^{q'}_t\cap L^{q_{\infty}'}_{t}}.
\end{align*}

\end{proof}

%%%%%%%%%%%%%%%%%%%%%%%%%%%%%%%%% 
\vskip 0.2in

\section{Large-time behavior of the volume function}\label{S:volume}

In this section we prove Theorems \ref{T:trB} and \ref{T:B}. 
We begin with an elementary consequence of the positivity of the matrix in \eqref{Ds}.

\begin{lemma}\label{L:log}
The function $t\to V(t)$ is strictly increasing on $(0,\infty)$.
\end{lemma}

\begin{proof}
It immediately follows from \eqref{mono} and  Weyl's monotonicity theorem (see \cite[Cor.8.4.10,(iv)]{Bernie}).

\end{proof}

Next, we make the trivial observation that, since $Q$ and $B$ are assumed to satisfy \textbf{(H)}, the same is true for $Q$ and $-B$. Since the matrix
\begin{equation}\label{tQ}
\tilde Q(t) = \int_0^t e^{-\tau B} Q e^{-\tau B^\star} d\tau,
\end{equation}
is the controllability Gramian associated of $Q$ and $-B$, we infer that it must be $\tilde Q(t)>0$ for every $t>0$. Making the change of variable $s = t-\tau$ in the definition of $Q(t)$ in \eqref{Ds}, we obtain
\begin{equation}\label{Qtilde}
Q(t) = e^{tB} \tilde Q(t) e^{tB^\star}.
\end{equation}
Taking the determinant of both sides of \eqref{Qtilde}, we reach the conclusion that 
\begin{equation}\label{trace2}
\frac{V(t)}{e^{2 t \tr B}} = \tilde V(t),
\end{equation}
where we have denoted $\tilde V(t) = \det \tilde Q(t)$.
The following lemma provides a basic consequence of \eqref{trace2}.

\begin{lemma}\label{L:ber}
The function $t\to \frac{V(t)}{e^{2 t \tr B}}$ is non-decreasing on $(0,\infty)$ and for $t\ge 1$ we have  
\begin{equation}\label{below}
\frac{V(t)}{e^{2 t \tr B}} \ge \tilde V(1)>0.
\end{equation}
\end{lemma}

\begin{proof}
From Lemma \ref{L:log}, applied to the matrices $Q$ and $-B$, we infer that $t\to \tilde V(t)$ is strictly increasing. This implies \eqref{below}.

\end{proof}

We will also need the following result.

%%%%%%%%%%%%%%%%%%%%%%%%%%%%%%%

%%%%%%%%%%%%%%%%%%%%%%%%%%%%%%%%%%%%%%

\begin{proposition}\label{P:ii}
Suppose that $\operatorname{Rank}(Q) = n$, and that $\sigma(B)\subset i \mathbb{R}$. There exist $D_{\infty}\ge n$ and $\gamma>0$, such that for $t\ge 1$ one has
\begin{align*}
V(t)\ge \gamma t^{D_{\infty}}.
\end{align*}
\end{proposition}

\begin{proof}
Since $Q>0$, there exist $\lambda>0$ such that $Q\ge \lambda I$. For any $\xi \in \Rn$ with $\xi\not= 0$, we have
\begin{align*}
\sa Q(t)\xi,\xi\da =\int_{0}^{t}\sa e^{sB}Qe^{sB^{\star}}\xi,\xi\da ds \ge \lambda \int_{0}^{t}\sa e^{sB}e^{sB^{\star}}\xi,\xi\da ds=\lambda \sa\left(\int_{0}^{t}e^{sB}e^{sB^{\star}} ds\right) \xi,\xi\da.
\end{align*}
By Weyl's monotonicity theorem, we infer that
\[
V(t) \ge \la^n \det \int_{0}^{t}e^{sB}e^{sB^{\star}} ds.
\]
Writing $B$ in its real Jordan form $B=PJP^{-1}$, see \cite[Theor.3.4.5]{HJ}, we obtain
\begin{align*}
\int_{0}^{t}e^{sB}e^{sB^{\star}}ds=P\left(\int_{0}^{t}e^{sJ}P^{-1}(P^{-1})^{\star}e^{sJ^{\star}}ds\right)P^{\star},
\end{align*}
which implies
\begin{align*}
V(t)\ge \la^n (\det P)^{2}\ \det \int_{0}^{t}e^{sJ}P^{-1}(P^{-1})^{\star}e^{sJ^{\star}}ds.
\end{align*}
Since $P^{-1}(P^{-1})^{\star}> 0$, arguing as before we see that there exists $\mu>0$ 
\begin{align*}
\det \int_{0}^{t}e^{sJ}P^{-1}(P^{-1})^{\star}e^{sJ^{\star}}ds\ge \mu^n \det \int_{0}^{t}e^{sJ}e^{sJ^{\star}}ds.
\end{align*}
These considerations show that, to complete the proof, it suffices to prove the existence of $D_\infty\ge n$ and $\gamma_0>0$, such that
\begin{equation}\label{hofm}
\det \int_{0}^{t}e^{sJ}e^{sJ^{\star}}ds\ \ge\ \gamma_0\ t^{D_\infty}\ \ \ \ \ \ t\ge 1.
\end{equation}
Since $\sigma(B)\subset i \mathbb{R}$, we have $\sigma(J)=\sigma(J^{\star})=\{0,\pm i\gamma_{1},...,\pm i\gamma_{p}\}$. Denoting with $n_1$ the multiplicity of the eigenvalue $\la = 0$, and with $m_\ell$ that of the eigenvalue $\la = i \gamma_\ell$, $\ell=1,...,p$, we have $n_{1}+2m_{1}+...+2m_{p}=n$. Since $J = \operatorname{diag}(J_{n_{1}},C_{m_1}(\gamma_1),...,C_{m_{p}}(\gamma_{p}))$, by \cite[Prop.11.2.8,(vi)]{Bernie}, we have 
\[
\det \int_{0}^{t}e^{sJ}e^{sJ^{\star}}ds = \det \int_0^t e^{s J_{n_1}}e^{s J_{n_1}^\star} ds\  \prod_{\ell=1}^p \det \int_0^t e^{s C_{m_{\ell}}(\gamma_{\ell})}e^{s C_{m_{\ell}}(\gamma_{\ell})^\star} ds.
\]
The proof will be completed if we show that there exist numbers 
\[
\gamma_{n_1}>0,\ \  d_{n_1}\ge n_1,\ \ \ \text{and}\ \ \ \ \ \ \gamma_{m_\ell}>0,\ \ d_{m_\ell}\ge m_\ell,\ \ \ \ell = 1,....,p,
\]
such that 
\begin{equation}\label{hwg}
\det \int_0^t e^{s J_{n_1}}e^{s J_{n_1}^\star} ds \ \ge\ \gamma_{n_1} t^{d_{n_1}},\ \ \ \ \det \int_0^t e^{s C_{m_{\ell}}(\gamma_{\ell})}e^{s C_{m_{\ell}}(\gamma_{\ell})^\star} ds\ \ge\ \gamma_{m_\ell}\ t^{2 d_{m_\ell}},\ \ \ \ t\ge 1.
\end{equation}
We recall that the $n_{1}\times n_{1}$ matrix $J_{n_{1}}$ and the $2m_{\ell}\times 2m_{\ell}$ matrices $C_{m_{\ell}}(\gamma_{\ell})$ are respectively in the form
\begin{align*}
J_{n_{1}}=\begin{pmatrix}
0 & 1 & 0 & \cdot & \cdot  & 0 \\
0 & 0 & 1 & \cdot  & \cdot &  0 \\
\cdot  & \cdot  & \cdot  & \cdot &  \cdot  & \cdot\\
\cdot & \cdot & \cdot & \cdot & \cdot   & \cdot\\
\cdot & \cdot & \cdot &\cdot & \cdot & 1 \\
0 & \cdot & \cdot & \cdot & \cdot &  0
\end{pmatrix}, \ \ \ C_{m_{\ell}}(\gamma_{\ell})=\begin{pmatrix}
C(\gamma_{\ell}) & I_{2} & 0 & \cdot & \cdot & \cdot & 0 \\
0 & C(\gamma_{\ell}) & I_{2} & 0 & \cdot & \cdot & 0 \\
\cdot & \cdot & \cdot & \cdot& \cdot & \cdot & \cdot \\
0 & \cdot & \cdot & \cdot & \cdot & \cdot &  I_{2} \\
0 & \cdot & \cdot & \cdot & \cdot & 0 & C(\gamma_{\ell})
\end{pmatrix},
\end{align*}
where 
\begin{align*}
C(\gamma_{\ell})=\begin{pmatrix}
0 & -\gamma_{\ell}\\
\gamma_{\ell} & 0
\end{pmatrix}, \ \ \ I_{2}=\begin{pmatrix}
1 & 0 \\
0 & 1
\end{pmatrix},
\end{align*}
see Theorem 3.4.1.5, (3.1.2) and (3.4.1.4)  in \cite{HJ}. 
Since the matrix $J_{n_{1}}$ is nilpotent, we have 
\begin{align*}
& \int_{0}^{t}e^{sJ_{n_{1}}}e^{sJ^{\star}_{n_{1}}}ds=\int_{0}^{t}\big(\sum_{i=0}^{n_{1}}\frac{(sJ_{n_{1}})^{i}}{i!}\big)\big(\sum_{i=0}^{n_{1}}\frac{(sJ_{n_{1}}^{\star})^{i}}{i!}\big)ds=tI_{n_{1}}+P^{(n_{1})}(t),
\end{align*}
where $P^{(n_{1})}(t)$ is a matrix whose entries are polynomials of degree at least $2$. We infer that
there exist $\gamma_{n_{1}}>0$ and $d_{n_{1}}\ge n_{1}$ such that if $t\ge 1$
\begin{align*}
	\det \int_{0}^{t}e^{sJ_{n_{1}}}e^{sJ_{n_{1}}^{\star}} ds\ \ge\ \gamma_{n_{1}}t^{d_{n_{1}}}.
\end{align*}
Next, by (11.2.15) in \cite{Bernie}, we can write 
\[
C_{m_{\ell}}(\gamma_{\ell}) =D_{m_{\ell}}(\gamma_{\ell})+N_{m_{\ell}}(\gamma_{\ell}),
\]
where $D_{m_{\ell}}(\gamma_{\ell})$ is skew-symmetric, $N_{m_{\ell}}(\gamma_{\ell})$ is nilpotent, and they commute. We thus have 
\[
e^{sC_{m_{\ell}}(\gamma_{\ell})}=e^{sN_{m_{\ell}}(\gamma_{\ell})}e^{sD_{m_{\ell}}(\gamma_{\ell})},
\] 
see \cite[Cor.11.1.6]{Bernie}. 
Since $D_{m_{\ell}}(\gamma_{\ell})$ is skew-symmetric, the matrix $e^{sD_{m_{\ell}}}$ is orthogonal, and we thus obtain
\begin{align*}
	\int_{0}^{t}e^{sC_{m_{\ell}}(\gamma_{\ell})}e^{sC_{m_{\ell}}^{\star}(\gamma_{\ell})}ds=&\int_{0}^{t}e^{sN_{m_{\ell}}(\gamma_{\ell})}e^{sN^{\star}_{m_{\ell}}(\gamma_{\ell})}ds=\\
	=&\int_{0}^{t}\big(\sum_{i=0}^{2m_{\ell}}\frac{(sN_{m_{\ell}}(\gamma_{\ell}))^{i}}{i!}\big)\big(\sum_{i=0}^{2m_{\ell}}\frac{(sN_{m_{\ell}}^{\star}(\gamma_{\ell}))^{i}}{i!}\big)ds=tI_{2m_{\ell}}+P^{(m_{\ell})}(t)>0
\end{align*}
where $P^{(m_{\ell})}(t)$ is a matrix whose entries are polynomials of degree at least $2$.
We conclude as before that there exist $\gamma_{m_{\ell}}>0$ and $d_{m_{\ell}}\ge m_{\ell}$ such that, if $t\ge 1$, 
\begin{align*}
\int_{0}^{t}e^{sC_{m_{\ell}}(\gamma_{\ell})}e^{sC_{m_{\ell}}^{\star}(\gamma_{\ell})}ds \ge \gamma_{m_{\ell}}t^{2d_{m_{\ell}}}.
\end{align*}
This completes the proof of \eqref{hwg}.

\end{proof}

%%%%%%%%%%%%%%%%%%%%

%%%%%%%%%%%%%%%%%%%%%%%%%%%%%%%%%%

We can now give the  

\vskip 0.2in

\begin{proof}[Proof of Theorem \ref{T:trB}]

We begin by noting that, based on \eqref{vicinoa0}, the estimate \eqref{H0} holds for $0<t\le 1$, regardless of the value of $\tr B$. We are thus left with proving that it is valid also for $t\ge 1$.
Assume now the hypothesis (i) in the statement of the theorem. There are three possible cases.

\medskip

\noindent \underline{Case (1)}: $\tr B<0$. If \eqref{quasi} holds, then for $t\ge 1$ the desired conclusion \eqref{H0} immediately follows  from \eqref{quasinfty} and the fact that $\tr B<0$. If instead there exists at least one $\la_0\in \sigma(B)$ such that $\Re \la_0\ge 0$, then by (iii) in Proposition \ref{P:boom0} we know that for some $c_0>0$ we have for $t\ge 1$
\[
V(t) \ge c_0 t.
\]
From this estimate we infer that the inequality \eqref{H0} holds for $t\ge 1$, provided that for some $C>0$
\[
e^{-t \tr B} \ge C t^{D-1}.
\]
Since $\tr B<0$, this is of course true.

\medskip

\noindent \underline{Case (2)}: $\tr B>0$. By appealing to \eqref{below} in Lemma \ref{L:ber}, we obtain for $t\ge 1$ 
\[
V(t) \ge \gamma_0\ e^{2 t \tr B},
\]
for some constant $\gamma_0>0$. Since we are assuming $\tr B>0$, it is clear from this estimate that, for some $\gamma>0$, the inequality \eqref{H0} holds true for every $t\ge 1$.

\medskip

\noindent \underline{Case (3)}: $\tr B = 0$. To complete the proof, we thus need to show that there exists $\gamma >0$ such that
\begin{equation}\label{op}
V(t) \ge \gamma\ t^D,\ \ \ \ \ \ \ t\ge 1.
\end{equation}
By the hypothesis $\sigma(B)\not\subset i \R$, there exists at least one $\la _0\in \sigma(B)$, such that $\Re \la_0 \not=0$. If $\Re \la_0>0$, then from (ii) in Proposition \ref{P:boom0} we have $V(t) \ge c_0 e^{2 \Re \la_0 t}$ for $t\ge 1$, and thus \eqref{op} follows. If instead $\Re \la_0<0$, then we claim that there must be at least one other $\la_1\in \sigma(B)$ such that $\Re \la_1>0$. Otherwise, if all other eigenvalues of $B$ had real part $\le 0$, we would have $\tr B < 0$, a contradiction. Again from (ii) in Proposition \ref{P:boom0} we have $V(t) \ge c_1 e^{2 \Re \la_1 t}$ for $t\ge 1$, and \eqref{op} follows. We have thus proved the theorem under the hypothesis (i).

Next, assume that hypothesis (ii) holds. Since $Q$ is invertible, by (b) following Definition \ref{D:hd}, we have $D = n$, and therefore \eqref{vicinoa0} implies
\[
V(t)\ge \gamma\ t^{n},\ \ \ \ \ 0\le t\le 1.
\]
Now, applying Proposition \ref{P:ii} we conclude that, for some $D_\infty \ge n$, we have
\begin{align*}
V(t)\ge \gamma\ t^{D_{\infty}},\ \ \ \ \ t\ge 1.
\end{align*} 
Combining the latter two inequalities, we reach the conclusion that \eqref{H0} holds for every $t\ge 0$.

Finally, assume that the hypothesis (iii) holds. In this case, the identity \eqref{voldil} gives
\[
V(t) = \gamma\ t^D,\ \ \ \ \ \ t>0,
\]
for some $\gamma >0$. Since, as we have noted, we must necessarily have $\tr B = 0$, we infer that \eqref{H0} holds in this case as well.

\end{proof}

%%%%%%%%%%%%%%%%%%%%%%%%%%%%%%%%%%%%%%%%%%%%%%%%%
The discussion that follows shows that, if in (ii) of Theorem \ref{T:trB} we substitute the assumption $\operatorname{Rank}(Q) = n$, with $\operatorname{Rank}(Q) < n$, the conclusion \eqref{H0} can fail significantly.

\medskip

\begin{example}[Failure of Hypothesis (A)]\label{E:imspec} \textnormal{Consider $2\times 2$ matrices $Q$ and $B$ as in \eqref{QB}, with $\operatorname{Rank}(Q)<2$, and $\sigma(B) \subset i \R$. Excluding the case $\sigma(B) = \{0\}$, the other possibility is $\sigma(B) = \{-i\gamma,i\gamma\}$ with $\gamma\not= 0$. Under this assumption, we can assume that
\[
Q = \begin{pmatrix} 1 & 0\\ 0 & 0\end{pmatrix},\ \ \ \ \ B = \begin{pmatrix} a & b\\ c & -a\end{pmatrix}, 
\]
with $a, b$ and $c$ satisfying the equation
\[
- bc = a^2 + \gamma^2 > 0.
\] 
The condition \textbf{(H)} imposes that $c\not= 0$, and therefore also $b\not= 0$, but the coefficient $a$ can be any real number. By the two-dimensional Rodrigues' formula, see e.g. \cite[Prop.11.3.2]{Bernie}, we have
\[
e^{sB} = \cos(s\gamma) I + \frac{\sin(s\gamma)}{\gamma} B.
\]
This easily gives
\[
e^{sB} Q e^{sB^\star} = \cos^2(s\gamma) Q + \frac{\sin(2s\gamma)}{2\gamma} (QB^\star + B Q) + \frac{\sin^2(s\gamma)}{\gamma^2} BQB^\star.
\]
Integrating this expression on $(0,t)$, after some elementary computations, we obtain
\begin{align}\label{2Vt}
Q(t) & = \frac 12 \left[t + \frac{\sin(2t\gamma)}2\right] Q + \frac 14\left[\frac 12 - \cos(2t\gamma)\right](QB^\star + B Q) + \frac 1{2\gamma^2} \left[t - \frac{\sin(2t\gamma)}2\right] BQB^\star
\\
& =  \frac t2 \left(C + \frac{A(t)}t\right),
\notag
\end{align}
where 
\begin{equation}\label{C}
C = Q + \frac 1{\gamma^2}BQB^\star 
= \begin{pmatrix} 1+\frac{a^2}{\gamma^2} & \frac{ac}{\gamma^2}\\ \frac{ac}{\gamma^2} & \frac{c^2}{\gamma^2}\end{pmatrix},
\end{equation}
and $A(t)$ is a matrix with bounded entries. Note that $\det C = \frac{c^2}{\gamma^2}>0$, therefore $C>0$. Since we will need a delicate generalization to higher dimension of this critical property  (see the proof of Theorem \ref{T:B}), we also give an alternative proof of the positivity of $C$. For every $x\in \R^2$, we have
\[
\sa Cx,x\da = |Q^{1/2}x|^2 + \frac 1{\gamma^2} |Q^{1/2} B^\star x|^2.
\]
Suppose that for some $x\not= 0$, we have $\sa Cx,x\da=0$. Then we must have both $x, B^\star x\in \operatorname{Ker}(Q)$, and therefore $E = \operatorname{span}\{x,B^\star x\} \subset \operatorname{Ker}(Q)$. But $E$ is a nontrival invariant subspace of $B^\star$. To see this, note that the Cayley-Hamilton theorem gives $(B^\star)^2 + \gamma^2 I = 0$, hence for any $\alpha, \beta \in \R$ we have 
\[
B^\star(\alpha x + \beta B^\star x) = \alpha B^\star x -\beta \gamma x\ \in E.
\]  
Since this contradicts \textbf{(H)}, we infer that it must be $\sa Cx,x\da>0$.
Now, note that \eqref{2Vt} implies the existence of $\gamma>0$ such that
\[
V(t) = \gamma t^2 \det(I + C^{-1} \frac{A(t)}t).
\]
Applying the formula (0.8.12) in \cite{HJ}, 
\begin{align}\label{det2}
\det (A+B)= \det A + \det B + \det A \tr(A^{-1}B),
\end{align}
with $A = I$ and $B = C^{-1} \frac{A(t)}t$, we reach the conclusion that
\begin{equation}\label{Vt22}
V(t) = \gamma t^2 (1 + O(t^{-1})),\ \ \ \ \ \text{as}\ t\to \infty.
\end{equation}
On the other hand, comparing with \eqref{QB} we see that, in the present example, we have $p_0 = \operatorname{Rank}(Q) = 1$, and that $B = \begin{pmatrix} \star_1 & \star_2\\ B_1& \star_3\end{pmatrix}$, with  
\[
\star_1 = (a),\ \ \star_2 = (b),\ \  B_1 = (c),\ \ \star_3 = (-a),
\]
so that $p_1 = \operatorname{Rank}(B_1) = 1$.
We conclude from \eqref{hd} that the local homogeneous dimension is $D = p_0 + 3 p_1 = 4$.
Therefore, in view of \eqref{Vt22} it is clear that \eqref{H0} cannot possibly hold. Instead, \eqref{HB} holds, with $D_\infty = 2$.}

\end{example}

%%%%%%%%%%%%%%%%%%%%%%%%%%%%%%%%%%%%%%%%%%%%%%%%%%%%%%%%
Theorem \ref{T:B} provides a generalization to arbitrary dimension of the large time behavior \eqref{Vt22} of the volume function. 

\vskip 0.2in
%%%%%%%%%%%%%%%%%%%%%%%%%%%%%%%%%%%%%%%

\begin{proof}[Proof of Theorem \ref{T:B}]
Since when $n=2$ we have already proved this result in Example \ref{E:imspec}, we can assume $n\ge 3$. 
As in \eqref{Vt22}, it suffices to show that, for some $\gamma>0$, one has for $t\to \infty$
\begin{align}\label{T:Bexp}
V(t)= \gamma t^{n}+O(t^{n-1}).
\end{align}
Since it will be relevant in the proof of \eqref{T:Bexp}, we recall some facts concerning the case $n=3$. 
Consider an arbitrary skew-symmetric matrix $B\in M_{3\times 3}(\R)$  
\begin{align*}
B=\begin{pmatrix}
0 & -c & b \\
c & 0 & -a \\
-b & a & 0
\end{pmatrix}.
\end{align*}
We have $\sigma(B)=\{0,\pm i \theta\}$, with $\theta=\sqrt{a^{2}+b^{2}+c^{2}}$, see (1.48) in \cite{McC}. Since the characteristic polynomial is 
\[
p_{B}(\lambda)=\lambda(\lambda^{2}+\theta^{2}),
\]
if $\theta>0$ the Cayley-Hamilton theorem gives
\begin{align*}
\frac{1}{\theta^3} B^3 =- \frac 1\theta B.
\end{align*}
This identity is the key to the formula of O. Rodrigues which states 
\begin{align}\label{rod}
e^{B}=I_{3}+\frac{\sin \theta}{\theta} B+\frac{1-\cos \theta}{\theta^2} B^{2},
\end{align}
see (2.14) in \cite{MuLiSa}. This  immediately gives 
\begin{align*}
e^{sB}=I_{3}+\frac{\sin \theta s}{\theta} B+\frac{1-\cos \theta s}{\theta^2} B^{2}.
\end{align*}

\noindent Moving back to the case of arbitrary $n\ge 3$, by the assumption on $B$ there exist an invertible matrix $P$, and a matrix $J$, such that $J^\star = - J$, 
for which
\begin{align*}
B=PJP^{-1}\ \Longrightarrow\ e^{sB}=Pe^{sJ}P^{-1}.
\end{align*}
Denoting by $\tilde{Q}=P^{-1}Q (P^{-1})^{\star}$, we clearly have $\tilde Q \ge 0$, and $\operatorname{Rank}(\tilde Q)<n$. Since moreover  
\begin{align*}
Q(t) & =\int_{0}^{t}e^{sB}Qe^{sB^{\star}}ds= P\int_{0}^{t}e^{sJ} \tilde Qe^{sJ^{\star}}ds  P^{\star},
\end{align*} 
we have
\[
V(t)=(\det P)^{2}\ \det \int_{0}^{t}e^{sJ}\tilde{Q}e^{sJ^{\star}}ds.
\]
We can thus focus on showing that \eqref{T:Bexp} holds when we assume from the start that $B = J$, i.e.,
\[
V(t) = \det \int_{0}^{t}e^{sJ} Q e^{sJ^{\star}}ds.
\]  
Now, unless $J = O_n$ (which is ruled out by the fact that, for \textbf{(H)} to hold, we would need $Q>0$, against the assumption), from (b) in \cite[Cor. 2.5.11]{HJ} we know that, for $q\ge 1$, the nonzero eigenvalues of $J$ are $\pm i \gamma_1,...,\pm i \gamma_q$, $\gamma_k>0$, $k=1,...,q$, and moreover there exists a real orthogonal matrix $U$ such that  
\begin{equation}\label{U}
U J U^{-1} = O_{n-2q} \oplus \gamma_1 \begin{pmatrix} 0 & -1\\ 1 & 0\end{pmatrix} \oplus . . . \oplus \gamma_q \begin{pmatrix} 0 & -1\\ 1 & 0\end{pmatrix}.
\end{equation}
If we indicate by $\theta_1,..., \theta_p$ the positive imaginary parts of the distinct eigenvalues of $J$, then, using \eqref{U}, it was proved in \cite[Theor. 2.2]{GaXu} that there exist $p$ unique skew-symmetric matrices $J_{1},...,J_{p}\in M_{n\times n}(\R)$, $2p\le n$, such that  
\begin{equation}\label{Deco2}
J_{i} J_{j}= J_{j} J_{i}=O_{n},\ \ \ \ \ i\neq j,
\end{equation}
\begin{equation}\label{Deco3}
J_{i}^{3}=-J_{i},\ \ \ \ \ \ i = 1,...,p,
\end{equation} 
\begin{equation}\label{Deco1}
J =\theta_{1}J_{1}+...+\theta_{p}J_{p},
\end{equation}
and, furthermore, the following generalization of Rodrigues' formula \eqref{rod} holds
\begin{align}\label{DecoExpo}
e^{J}=I_{n}+\sum_{i=1}^{p}(\sin \theta_{i} J_{i}+(1-\cos \theta_{i})^{2}J_{i}^{2}).
\end{align}
Before proceeding, if $x\in \Rn\setminus\{0\}$, consider the subspace
\[
V(x) = \operatorname{span}\{J_{1}x,J_{1}^{2}x,J_{2}x,J_{2}^{2}x,...,J_{p}x,J_{p}^{2}x\}.
\] 
We claim that:
\begin{equation}\label{VV}
J^\star(V(x))\subset V(x),
\end{equation}
i.e., $V(x)$ is a nontrivial invariant space of $J^\star$. To see \eqref{VV}, let $y\in V(x)$. There exist $ \alpha_1,\beta_1,...,\alpha_p,\beta_p\in \R$ such that
\[
y = \alpha_1 J_1 x + \beta_1 J_1^2 x + ... + \alpha_p J_p x + \beta_p J_p^2 x.
\]
This gives
\begin{align*}
J^\star y & = \sum_{i,j=1}^p \theta_i \alpha_j  J^\star_i J_j x + \sum_{i,j=1}^p \theta_i \beta_j  J^\star_i J^2_j x 
\\
& = - \sum_{i,j=1}^p \theta_i \alpha_j  J_i J_j x - \sum_{i,j=1}^p \theta_i \beta_j  J_i J^2_j x 
\\
& = - \sum_{i=1}^p \theta_i \alpha_i  J_i^2 x - \sum_{i=1}^p \theta_i \beta_i  J_i^3 x,
\end{align*}
where we have used \eqref{Deco2}. By \eqref{Deco3}, we conclude that 
\[
J^\star y = - \sum_{i=1}^p \theta_i \alpha_i  J_i^2 x + \sum_{i=1}^p \theta_i \beta_i  J_i x\ \in V(x),
\]
which proves \eqref{VV}. Next, we note that, in view of \textbf{(H)} (which we are assuming for the matrix $J$), the inclusion \eqref{VV} implies that 
\begin{equation}\label{in}
V(x) \not\subset \ker Q,\ \ \ \ \forall x\in \Rn\setminus\{0\}.
\end{equation}
For the sequel, it is important to note that \eqref{in} implies that
\begin{equation}\label{inn}
\sum_{i=1}^{p}(|Q^{1/2} J_{i}x|^2 + |Q^{1/2} J_i^2 x|^2) > 0\ \ \ \ \forall x\in\Rn\setminus\{0\}.
\end{equation}
Otherwise, we would have $Q J_i x = Q J_i^2 x = 0$ for $i=1,...,p$, and therefore $V(x)\subset \ker Q$.
For subsequent purposes, we introduce the notation 
\[
\hat{J}=J_{1}+...+J_{p}.
\] 
We observe that $\hat J^\star = - \hat J$, and that, in view of \eqref{Deco2}, we have
\begin{equation}\label{hat2}
\hat{J}^{2}=J_{1}^{2}+...+J_{p}^{2},\ \ \ \ (\hat{J}^\star)^{2}=(J^\star_{1})^{2}+...+(J^\star_{p})^{2}.
\end{equation}
To continue, note that \eqref{DecoExpo} implies
\begin{align*}
e^{sJ}=I_{n}+\sum_{i=1}^{p}(\sin (\theta_{i}s)J_{i}+(1-\cos(\theta_{i}s))J_{i}^{2}).
\end{align*}
Consequently, we have
\begin{align}\label{gnam}
 e^{sJ}Qe^{sJ^{\star}}  = Q &+  \sum_{i=1}^{p}\bigg\{\sin(\theta_{i}s)J_{i}Q+(1-\cos(\theta_{i}s))J_{i}^{2}Q + \sin(\theta_{i}s)QJ_{i}^{\star}+(1-\cos(\theta_{i}s))Q(J_{i}^{\star})^{2}\bigg\}
\\
&+\sum_{i,j=1}^{p}\bigg\{\sin(\theta_{i}s)\sin(\theta_{j}s)J_{i}QJ_{j}^{\star}+(1-\cos(\theta_{i}s))(1-\cos(\theta_{j}s))J_{i}^{2}Q(J_{j}^{\star})^{2}\bigg\}
\notag\\
&+\sum_{i,j=1}^{p}\bigg\{\sin(\theta_{i}s)(1-\cos(\theta_{j}s))J_{i}Q(J_{j}^{\star})^{2}+(1-\cos(\theta_{i}s))\sin(\theta_{j}s)J_{i}^{2}QJ_{j}^{\star}\bigg\}.
\notag
\end{align}
Integrating both sides of \eqref{gnam} in $s\in [0,t]$, using \eqref{hat2} and observing that (see \eqref {2Vt}) 
\[
\int_{0}^{t}\sin^{2}(\theta_{i}s)ds=\frac{t}{2} +O(1),\ \ \ \ \ \ \int_{0}^{t}(1-\cos(\theta_{i}s))^{2}ds=\frac{3t}{2}+O(1),
\]
we reach the conclusion that 
\begin{align}\label{QQt}
Q(t)& = tQ+t\hat{J}^{2}Q+tQ(\hat{J}^{\star})^{2}+t\sum_{i=1}^{p}\left\{\frac{1}{2} J_{i} Q J_{i}^{\star}+ \frac{3}{2}J_{i}^{2}Q (J_{i}^{\star})^{2}\right\}+t\sum_{\substack{i,j=1,\\i\neq j}}^{p}J_{i}^{2}Q(J_{j}^{\star})^{2}+D(t)
\\
& = tQ+t\hat{J}^{2}Q+tQ(\hat{J}^{\star})^{2}+t \sum_{i,j=1}^{p}J_{i}^{2}Q(J_{j}^{\star})^{2}+\frac{t}{2}\sum_{i=1}^{p}\left\{J_{i} Q J_{i}^{\star}+ J_{i}^{2}Q (J_{i}^{\star})^{2}\right\}+D(t)
\notag\\
& = tQ+t\hat{J}^{2}Q+tQ(\hat{J}^{\star})^{2}+t\hat{J}^{2}Q(\hat{J}^{\star})^{2}+\frac{t}{2}\sum_{i=1}^{p}\left\{J_{i} Q J_{i}^{\star}+ J_{i}^{2}Q (J_{i}^{\star})^{2}\right\}+D(t)
\notag\\
& = t C  + D(t), 
\notag
\end{align}
where
\begin{equation}\label{CC}
C = Q+\hat{J}^{2}Q+Q(\hat{J}^{\star})^{2}+\hat{J}^{2}Q(\hat{J}^{\star})^{2}+\frac{1}{2}\sum_{i=1}^{p}(J_{i} Q J_{i}^{\star}+ J_{i}^{2}Q (J_{i}^{\star})^{2}),
\end{equation}
and we have denoted by $D(t)$ a matrix with bounded entries. We now make the following crucial

\vskip 0.2in

\noindent \underline{Claim:} $C>0$.

\vskip 0.2in

\noindent To prove the claim, we show that for any $x\in \Rn\setminus\{0\}$, we must have
\begin{equation}\label{Cx}
\sa Cx,x\da > 0.
\end{equation}
From \eqref{CC} we have
\begin{align}\label{CCx}
\sa Cx,x\da & = |Q^{1/2}x|^2 - 2 \sa\hat{J}Qx,\hat J x\da  + |Q^{1/2} \hat J^2 x|^2 + \frac{1}{2}\sum_{i=1}^{p}(|Q^{1/2} J_{i}x|^2 + |Q^{1/2} J_i^2 x|^2)
\\
& = |Q^{1/2}x|^2 + 2 \sa Q^{1/2} x,Q^{1/2}\hat J^2 x\da  + |Q^{1/2} \hat J^2 x|^2 + \frac{1}{2}\sum_{i=1}^{p}(|Q^{1/2} J_{i}x|^2 + |Q^{1/2} J_i^2 x|^2).
\notag
\end{align}
There exist two possibilities:

\noindent \underline{Case (i):}  $x\in \ker Q$. In such case, we obtain from \eqref{CCx}
\[
\sa Cx,x\da = |Q^{1/2} \hat J^2 x|^2 + \frac{1}{2}\sum_{i=1}^{p}(|Q^{1/2} J_{i}x|^2 + |Q^{1/2} J_i^2 x|^2) > 0,
\]
where the last inequality is justified by \eqref{inn}.

\noindent \underline{Case (ii):} $x\not\in \ker Q$. The Cauchy-Schwarz inequality gives
\[
- 2 \sa Q^{1/2} x,Q^{1/2}\hat J^2 x\da \le 2 |Q^{1/2} x| |Q^{1/2}\hat J^2 x|.
\]
Using this inequality in \eqref{CCx}, we obtain
\[
\sa Cx,x\da \ge \left(|Q^{1/2} x| - |Q^{1/2}\hat J^2 x|\right)^2 + \frac{1}{2}\sum_{i=1}^{p}(|Q^{1/2} J_{i}x|^2 + |Q^{1/2} J_i^2 x|^2) > 0,
\]
where, again, the last inequality follows from  \eqref{inn}. We have thus established \eqref{Cx}.
 
Returning to \eqref{QQt}, we see that the claim allows to represent
\begin{equation}\label{QQtt}
Q(t) = t C \left[I + (tC)^{-1} D(t)\right].
\end{equation}
To finish, we apply the formula 
\begin{equation}\label{detI}
\det(I+X) = \sum_{k=0}^n \tr \text{adj}_k(X),
\end{equation}
valid for any $X\in M_{n\times n}(\R)$ (see \cite[p. 29]{HJ}), with the choice $X= (tC)^{-1} D(t)$. Here, for $k=0,1,...,n$, we have let $n(k) = \begin{pmatrix} n \\ k\end{pmatrix}$, and denoted by $\text{adj}_k(X)\in M_{n(k)\times n(k)}(\R)$ the  $k$-th adjugate matrix of $X$, see \cite[(0.8.12.1)]{HJ}. Keeping in mind that 
\[
\operatorname{adj}_n(X) = 1,\ \ \ \operatorname{adj}_0(X) = \det X,\ \ \  \operatorname{adj}_1(X) = X,
\]
and noting that, for $1\le k \le (n-1)$,
\begin{align*}
\operatorname{adj}_{k}((tC)^{-1}D(t))=\frac{1}{t^{n-k}}\operatorname{adj}_{k}(C^{-1}D(t)) \implies \tr(\operatorname{adj}_{k}(\frac{1}{t}C^{-1}D(t)))=O(\frac{1}{t^{n-k}}), 
\end{align*}
we finally obtain from \eqref{QQtt} and \eqref{detI}
\begin{align*}
V(t)=\gamma t^{n}+ O(t^{n-1}),
\end{align*}
where $\gamma = \det C>0$. This proves \eqref{T:Bexp} when $B = J$, and we are thus finished.

\end{proof}

%%%%%%%%%%%%%%%%%%%%%%%%%%%%%%%%%%%%%%%%%%

%%%%%%%%%%%%%%%%%%%%%%%%%%%%%%%%%%%%%%%%%%

%%%%%%%%%%%%%%%%%%%%%%%%%%%%%%%%%%%%%%%%%%%%%%%%%%%%%%

%%%%%%%%%%%%%%%%%%%%%%%%%%%%%%%%%%%%%%%%

\vskip 0.2in

%%%%%%%%%%%%%%%%%%%%%%%%%%%%%%%%%%%%%%%%

\section{Examples and final comments}\label{S:example} 

The purpose of this final section is to illustrate the scope of our results with some significant examples. We begin with the following generalization of the problem \eqref{quattroo} in Section \ref{S:exp}.

\begin{proposition}\label{E:inv}
In $\Rn\times \R$, consider the Cauchy problem
\begin{equation}\label{I}
\p_t u - i \Delta u -  \sa Bx,\nabla u\da   = F(x,t),\ \ \ \ \ \ \ u(x,0) = \vf(x).
\end{equation}
Fix $r>2$, such that $r<\frac{2n}{n-2}$ when $n\ge 3$, and consider a pair $(q,r)$ satisfying $\frac 2q = n\left(\frac 12 - \frac 1r\right)$. Then one has for the 
solution $u$ of \eqref{I}  
\begin{align}\label{stri4}
& \bigg(\int_\R e^{\frac{q\tr B}2 t}  \big(\int_{\Rn}|u(x,t)|^r dx\big)^{\frac qr} dt\bigg)^{\frac 1q} \le C \bigg\{\|\vf\|_{L^2(\Rn)}
 + \bigg(\int_\R e^{\frac{q'\tr B}2 t}  \big(\int_{\Rn}|F(x,t)|^{r'} dx\big)^{\frac{q'}{r'}} dt\bigg)^{\frac 1{q'}}\bigg\}.
\end{align}
\end{proposition}

\begin{proof}
Since $\operatorname{Rank}(Q) =n$, according to Definition \ref{D:hd} the local homogeneous dimension of \eqref{I} is $D = n$. Therefore, the admissibility range from \eqref{admissible} is the same as in \cite{GVstrich}, and the estimate \eqref{stri4} follows from Theorem A. When, in particular, $B = - I_n$, we recover problem \eqref{quattroo}.
If instead the matrix $B$ satisfies one of the hypothesis (i), (ii) in Theorem \ref{T:trB}, then $\tr B = 0$, and therefore \eqref{stri4} gives for \eqref{I} the same estimate that one would obtain for a zero drift.

\end{proof}

\medskip

%%%%%%%%%%%%%%%%%%%%%%%%%%%%%%
Next, we discuss an interesting family of problems in which the Hamiltonian displays a varying degeneracy. Fix $n\ge 2$ and for any $1\le k \le n$, write points $x = (x_1,...,x_n)\in \Rn$ as $(\hat x_k,x_{k+1},...,x_n)$, where $\hat x_k = (x_1,...,x_k)$. Let $\hat \Delta_{k} = \p_{x_1x_1} + ... + \p_{x_k x_k}$ represent the Laplacian with respect to the coordinates $\hat x_k$. 
 
\begin{proposition}\label{E:fan} For each $k=1,...,n$, consider the Cauchy problem for the Schr\"odinger equation in $\Rn\times (0,\infty)$  
\begin{equation}\label{bad}
\p_t u = i \hat \Delta_k u + x_1 \p_{x_2} u + . . . + x_{n-1} \p_{x_n} u + F(x,t),\ \ \ \ \ u(x,0) = \vf(x).
\end{equation}
Then the local homogeneous dimension of \eqref{bad} is given by the number
\begin{equation}\label{Dk}
D_n(k) = n + (n-k+1)(n-k),\ \ \ \ \ \ 1\le k\le n,
\end{equation}
with corresponding \emph{Strichartz pairs} 
\begin{equation}\label{stricouples}
r_n(k) = \frac{2(D_n(k)+2)}{D_n(k)} \nearrow,\ \ \ \ \ r'_n(k) = \frac{2(D_n(k)+2)}{D_n(k)+4}\searrow,\ \ \ \ \ \ 1\le k\le n.
\end{equation}
Let $q_n(k)$ be the exponent defined by the equation
\begin{equation}\label{kad}
\frac{2}{q_n(k)} = D_n(k)\left(\frac 12 - \frac{1}{r_n(k)}\right),\ \ \ \ \ \ 1\le k\le n.
\end{equation}
Then the following Strichartz estimate holds
\begin{align}\label{stribad}
& \bigg(\int_\R ||u(\cdot,t)||^{q_n(k)}_{L^{r_n(k)}(\Rn)} dt \bigg)^{\frac 1{q_n(k)}} \le C \bigg\{\|\vf\|_{L^2(\Rn)} + \bigg(\int_\R ||F(\cdot,t)||^{q'_n(k)}_{L^{r'_n(k)}(\Rn)} dt \bigg)^{\frac 1{q'_n(k)}}\bigg\}.
\end{align}
\end{proposition}

\begin{proof}
The differential equation in \eqref{bad} can be represented in the form \eqref{A}, if we take
\begin{equation}\label{QkB}
Q = Q(k) = \begin{pmatrix} I_{k} & O_{k\times (n-k)}\\O_{(n-k)\times k} & O_{(n-k)\times(n-k)}\end{pmatrix},\ \ \ B = \begin{pmatrix} 0 & 0 & \cdot & \cdot & \cdot & 0 & 0
\\
1 & 0 & \cdot & \cdot & \cdot & 0 & 0
\\
0 & 1 & \cdot & \cdot & \cdot &  0 & 0
\\
\cdot & \cdot & \cdot & \cdot & \cdot & \cdot & \cdot
\\
\cdot & \cdot & \cdot & \cdot & \cdot & \cdot & \cdot
\\
\cdot & \cdot & \cdot & \cdot & \cdot & \cdot & \cdot
\\
0 & 0 & \cdot & \cdot & \cdot &  1 & 0
\end{pmatrix}.
\end{equation}
Note that $\sigma(B) = \{0\}$, $B$ is nilpotent since $B^n = O_n$, and $B^\star e_j = e_{j-1}$ for $j=1,...,n$. Therefore, the only nontrivial invariant subspaces of $B^\star$ are
\[
E_j = \operatorname{span}\{e_1,...,e_j\},\ \ \ \ \ \ j=1,...,n-1.
\]
Since $\operatorname{Ker} Q(k) = \operatorname{span}\{e_{k+1},...,e_n\}$, it is clear that the matrices $Q(k)$ and $B$ satisfy the hypothesis \textbf{(H)}. To compute the homogeneous dimension of \eqref{bad}, for each $1\le k\le n$, we need to write $B$ in the canonical form \eqref{QB} corresponding to $Q(k)$. We start by considering $k=n$. In this case we clearly have $p_{0}=n$ and $B_{1}=B$. Since $p_{j}=0$ if $j>1$, we find
\begin{align*}
D_{n}(n)=n.
\end{align*}
If instead $k\in\{1,...,n-1\}$, we write
\begin{align}\label{Bbad}
B=\begin{pmatrix}
\star_{(1)} & \star_{(2)} & \star_{(3)} & \cdot &\star_{(n-k)} &  \star_{(n-k+1)} \\
B_{1} & 0 &  \\
0 & B_{2} & 0 \\
\cdot & \cdot & \cdot \\
0 & \cdot & \cdot &  0 & B_{n-k} & 0
\end{pmatrix},
\end{align}
where 
$\star_{(1)}$ is the $k\times k$ matrix 
\begin{align*}
\star_{(1)}=\begin{pmatrix}
0 & 0 & \cdot & \cdot & 0 \\
1 & 0 & \cdot & \cdot & 0 \\
\cdot & \cdot & \cdot & \cdot & \cdot \\
0 & \cdot & \cdot &  1 & 0
\end{pmatrix},
\end{align*}
$\star_{(\ell)}$ is the $(k+\ell-2)\times 1$ zero matrix, and
\begin{align*}
B_{1}=\begin{pmatrix}
0 & \cdot & \cdot & \cdot & 1
\end{pmatrix} \in M_{1\times k}(\R),\ \ \ \  B_{2}=...=B_{n-k}= (1)\in M_{1\times 1}(\R).
\end{align*}
Consequently, we have $p_{0}=k$, $p_{1}=p_{2}=...=p_{n-k}=1$, and $p_{j}=0$ if $j>n-k$. It thus follows that 
\begin{align*}
D_{n}(k)&=k+3p_{1}+5p_{2}+...+(2(n-k)+1)p_{n-k}=\\
&=k+\sum_{j=1}^{n-k}(2j+1)=n+(n-k+1)(n-k),
\end{align*}
which proves \eqref{Dk}. Once the local homogeneous dimension $D_n(k)$ is known, the Strichartz pairs \eqref{stricouples}, and the admissibility pairs $(q_n(k),r_n(k))$ in \eqref{kad} are respectively determined by \eqref{rr} and \eqref{admissible}. To complete the proof, we are left with showing \eqref{stribad}. This is a delicate part, as we need to understand whether \eqref{H0} or \eqref{HB} are valid, and consequently whether Theorem A or B applies. For $1\le k\le n$, we denote 
\[
V_k(t) = \det Q_k(t) = \det \int_0^t e^{sB} Q(k) e^{sB^\star} ds.
\]
When $k = n$, the corresponding equation in \eqref{bad} is non-degenerate. Since from \eqref{Dk} we know that $D_n(n) = n$, and since $\tr B = 0$, Proposition \ref{P:ii} guarantees that \eqref{H0} is valid for the volume function $V_n(t)$. We infer that Theorem A applies, and \eqref{stribad} is valid with the same admissible pairs of the equation without drift.
When $k<n$, the quadratic form in \eqref{bad} becomes increasingly singular as $k\searrow 1$. The case $k=1$ is the most degenerate, corresponding to 
\begin{equation}\label{badd}
\p_t u = i \p_{x_1x_1} u + x_1 \p_{x_2} u + . . . + x_{n-1} \p_{x_n} u + F(x,t),\ \ \ \ \ u(x,0) = \vf(x).
\end{equation}  
Interestingly, this case is the easiest to decide since, within the family of problems \eqref{bad}, it is the only one in which, when $F=0$, the corresponding differential operator admits a global homogeneous structure given by the non-isotropic dilations  
\begin{equation}\label{badil}
\delta_\lambda(x,t) = \left(\lambda
x_1,\lambda^3
x_2,\ldots,\lambda^{2n-1}x_n,\lambda^2t\right).
\end{equation}
This follows from the fact that, when $k=1$, all the $\star$ matrices in  the canonical representation \eqref{Bbad} are zero matrices, whereas $B_1 = (1)\in M_{1\times 1}(\R)$, and therefore we are in the situation of hypothesis (iii) in Theorem \ref{T:trB}. Therefore, Theorem A applies, resulting in \eqref{stribad}, with
\begin{equation}\label{Dmaxx}
D = D_n(1) = \sum_{k=1}^{n} (2k - 1) = n^2.
\end{equation}
We observe that 
\begin{equation}\label{stack}
D_n(1)  = n^2\ge ...\ge D_n(n) = n.
\end{equation}
We remark that $D$ in \eqref{Dmaxx} is the spatial \emph{homogeneous dimension} of \eqref{badil}, in the sense of  \cite{RS}, \cite{FS}, \cite{NSW}.
It remains to decide the cases $2\le k\le n$. We already know from \eqref{voldil} that
\[
V_1(t) = V_1(1)\ t^{n^2}, \ \ \ \ \ \ t>0.
\]

\vskip 0.2in

\noindent \underline{Claim:} For $2\le k\le n$ we have as $t\to \infty$
\begin{align}\label{detQinfDym}
V_k(t) = V_{1}(1)\ t^{n^{2}}(1+o(1)).
\end{align}

\vskip 0.2in

\noindent The claim can be extracted from the proof of \cite[Ex.2, p.287]{Ku2}, but for the sake of completeness we provide some details. The matrix $B$ is the transposed of the Jordan canonical form of a non diagonalizable matrix with zero eigenvalues. It is immediate to check, see (11.2.18) in \cite{Bernie}, that
\begin{align*}
e^{sB}=\begin{pmatrix}
1 & 0 & \cdot & \cdot & \cdot & \cdot & 0 \\
s & 1 & \cdot & \cdot & \cdot & \cdot &  0 \\
\frac{s^{2}}{2!} & s & 1 & 0 & \cdot & \cdot  & 0\\
\frac{s^{3}}{3!} & \frac{s^{2}}{2!} & s & 1 & 0  & \cdot & 0\\
\cdot & \cdot & \cdot & \cdot & \cdot & \cdot & \cdot\\
\frac{s^{n-1}}{(n-1)!} & \cdot & \cdot & \cdot & \cdot & \cdot& 1
\end{pmatrix}.
\end{align*}
Let $e_1,...,e_n$ be the unit (column) vectors of the standard basis of $\Rn$. Observe that we can express the Gramian matrix associated with $Q$ and $B$ as in \eqref{QkB}, in the following way
\begin{align*}
Q(t)&=\int_{0}^{t}e^{sB}Qe^{sB^{\star}}ds=\int_{0}^{t}e^{sB}e_{1}e_{1}^{\star}e^{sB^{\star}}ds+...+\int_{0}^{t}e^{sB}e_{k}e_{k}^{\star}e^{sB^{\star}}ds=\\
&=Q_{1}(t)+...+Q_{k}(t).
\end{align*}
A direct computation shows that for $1\le r \le k$
\begin{align*}
Q_{r}(t)=\begin{pmatrix}
0_{(r-1)\times(r-1)} & 0_{(r-1)\times(n-r+1)}\\
0_{(n-r+1)\times(r-1)} & A_{r}(t)
\end{pmatrix},
\end{align*}
where $A_{r}(t)$ is a $(n-r+1)\times(n-r+1)$ matrix whose entires are
\begin{align*}
[A_{r}(t)]_{i,j}=\frac{t^{i+j-1}}{(i+j-1)(i-1)!(j-1)!}, \ \ \ \ i,j=1,...,n-r+1.
\end{align*}
For any $h,\ell \in \{1,...,n\}$ and for any $r\in \{2,...,k\}$ we cleary have
\begin{align*}
\lim\limits_{t\to \infty} \frac{[Q_{r}(t)]_{h\ell}}{[Q_{1}(t)]_{h\ell}}=0.
\end{align*}
As a consequence, we obtain
\begin{align*}
V_k(t) &= \sum_{i_{1},i_{2},...,i_{n}} \operatorname{sgn}\{i_{1},...i_{n}\}[Q(t)]_{1,i_{1}}...[Q(t)]_{n,i_{n}}=\\
&=t^{n^{2}}(1+o(1))\sum_{i_{1},i_{2},...,i_{n}} \operatorname{sgn}\{i_{1},...i_{n}\}[Q_{1}(1)]_{1,i_{1}}...[Q_{1}(1)]_{n,i_{n}}
\\
&=t^{n^{2}}V_{1}(1)(1+o(1)), 
\end{align*}
which establishes \eqref{detQinfDym}.
We emphasize that the claim proves that, for some $\gamma_k>0$, and with $D$ as in \eqref{Dmaxx}, one has
\[
V_k(t)\ge \gamma_k\ t^{D},\ \ \ \ \ \ \ t>0.
\]
From this estimate, and from \eqref{stack}, we conclude that \eqref{H0} is always valid, and therefore we can apply Theorem A, thus establishing \eqref{kad} for any of the problems \eqref{bad}, with $k=1,...,n$.
This finishes the proof.

\end{proof}

Proposition \ref{E:fan} shows that, despite the fact that, when $2\le k\le n$, the equation \eqref{bad} lacks of a global homogeneous structure, for each $k$ it is possible to determine a \emph{local homogeneous dimension} $D_n(k)$ which decides what Strichartz estimates to expect. 
We note, in particular (see \eqref{rr})  that if $\vf\in L^2(\Rn)$ and the forcing term $F\in L^{r'_n(k)}(\R^{n+1})$, then the solution $u$ to  \eqref{bad} belongs to $L^{r_n(k)}(\R^{n+1})$, and moreover there exists a universal $C(n,k)>0$ such that the following Strichartz estimate holds  
\begin{equation}\label{strichartzk}
||u||_{L^{r_n(k)}(\R^{n+1})} \le C(n,k) \left(||\vf||_{L^2(\Rn)} + ||F||_{L^{r'_n(k)}(\R^{n+1})}\right).
\end{equation}
The reader should note that, as $k\searrow 1$, an increasing integrability is required for the forcing term $F$ (i.e., $L^{r_n'(k)}$ instead of $L^{\frac{2(n+2)}{n+4}}$), while the solution $u$ experiences a corresponding increasing loss of integrability (i.e., $L^{r_n(k)}$ instead of $L^{\frac{2(n+2)}{n}}$).
This phenomenon of anomalous dispersion results from the strong degeneracy of the quadratic form in \eqref{bad}.
We finally note that, when $k=1$, the global dilations \eqref{badil} imply that the exponents in \eqref{strichartzk} are best possible. 

\medskip

\noindent As a final comment, we emphasize that estimates such as \eqref{strichartzk}, or the more general \eqref{stribad}, are highly sensitive to the spectrum 
$\sigma(B)$ of the drift matrix. Suppose, in fact, we change the matrix $B$ in \eqref{bad} by substituting the $0$ in the upper left-corner with a $\pm 2$. This changes $\sigma(B)$ from $\{0\}$, to $\{0,\pm 2\}$, with the eigenvalue $\la = 0$ having multiplicity $n-1$. If we consider the case $k=1$ of \eqref{badd}, then this change does not affect the local homogeneous dimension $D_n(1) = n^2$, and consequently the admissibility pairs $(q_n(1),r_n(1))$ remain unchanged. But the resulting estimate from \eqref{strichartzone} in Theorem A now contains an exponential weight, since $\tr B = \pm 2$. We thus have 
\begin{align}\label{strichartz2}
& \bigg(\int_\R e^{\pm q t} \big(\int_{\Rn}|u(x,t)|^r dx\big)^{\frac qr} dt\bigg)^{\frac 1q} \le C \bigg\{\|\vf\|_{L^2(\Rn)}
 + \bigg(\int_\R e^{\pm q' t} \big(\int_{\Rn}|F(x,t)|^{r'} dx\big)^{\frac{q'}{r'}} dt\bigg)^{\frac 1{q'}}\bigg\},
\end{align}
where for notational ease we have let $r = r_n(1), q = q_n(1)$, and $r', q'$ indicate their dual.

%%%%%%%%%%%%%%%%%%%%%%%%%%%%%%%%%%%%%%%%%%%%%%%%

%%%%%%%%%%%%%%%%%%%%%%%%%%%%%%%%%%%%%%%%%%%%%%%%%%%%%%%%%%%%

%%%%%%%%%%%%%%%%%%%%%%%%%%%%%%%%%%%%%%%%%%%%%%%%%%%%%%%%%%%%%%%%

\vskip 0.2in

%%%%%%%%%%%%%%%%%%%%%%%%%%%%%%%%%%%%%%%%%%%%%%%%%%%%%%%%%%%%%%%%

\begin{example}\label{E:kolmo} \textnormal{Another interesting example which falls within the scope of Theorem A is the Schr\"odinger equation associated with the degenerate Kolmogorov operator in \cite{Kol}
\begin{equation}\label{degschr}
\p_t u - i \Delta_x u - \sa x,\nabla_y u\da = F(x,y,t),\ \ \ \ u(x,y,0) = \vf(x,y),
\end{equation}
where we have let $x, y\in \R^m$, so that $(x,y)\in \Rn$ with $n = 2m$. For \eqref{degschr} we have 
\[
Q = \begin{pmatrix} I_m & O_m\\ O_m & O_m\end{pmatrix},\ \ \ \  B = \begin{pmatrix} O_m & O_m\\I_m & O_m\end{pmatrix},
\]
so that $B$ is in the form \eqref{barB}, and comparing with \eqref{QB} we see that $Q_0  = B_1 = I_m$, so that $p_0 = p_1 = m$. We thus obtain from Definition \ref{D:hd}
\[
D = m + 3 m = 2n,
\]
twice the dimension of the ambient space. In this case \eqref{strichartze} implies the Strichartz estimate   
\begin{equation}\label{stridegS}
||u||_{L^{\frac{2(n+1)}n}(\R^{n+1})} \le C(n) \left(||\vf||_{L^2(\Rn)} + ||F||_{L^{\frac{2(n+1)}{n+2}}(\R^{n+1})}\right),
\end{equation}
whose best possible character follows from the invariance of the operator in the left-hand side of \eqref{degschr} with respect to the scaling $\delta_\la(x,y,t) = (\la x,\la^3 y, \la^2 t)$.}
\end{example}

%%%%%%%%%%%%%%%%%%%%%%%%%%%%%%%%%%%%%%%%%%%%%%%%%% 

\vskip 0.2in

\begin{example}[The anomalous case]\label{E:B}
\textnormal{This final example serves to illustrate the situation discussed in the closing of Section \ref{S:truth}. In $\R^4$ consider the matrices 
\begin{align*}
Q_{k}=\begin{pmatrix}
I_{k} & O_{k\times(4-k)} \\
O_{(4-k)\times k} & O_{(4-k)\times (4-k)}
\end{pmatrix},\ \ \ \ \ \ \ 			B=\begin{pmatrix}
0 & 0 & 0 & 0 \\ 
0 & 0  & 0 & -1 \\ 
1 & 0 & 0 & 0 \\
0 & 1 & 0 & 0
\end{pmatrix},\ \ \ \ k=2,3.
\end{align*}
The matrix $B$ has eigenvalues
\[
\lambda_{1}=0,\ \ \ \ \ \ \lambda_{2}=i,\  \ \ \ \ \ \lambda_{3}=-i,
\]
with corresponding eigenvectors
\begin{align*}
v_{1}=\begin{pmatrix}
0 \\ 0 \\ 1 \\ 0
\end{pmatrix}; \ \ \ 
v_{2}=\begin{pmatrix}
0 \\ i \\ 0 \\ 1
\end{pmatrix}; \  \ \
v_{2}=\begin{pmatrix}
0 \\ -i \\ 0 \\ 1
\end{pmatrix}.
\end{align*}
In summary, we have
\begin{itemize}
\item $\operatorname{Rank}(Q_k)<n\ (=4)$; 
\item $\sigma(B)\subset i \R$, and $B$\ is not similar to a skew-symmetric matrix; 
\item $B$ is not in the form \eqref{barB},
\end{itemize}
and therefore the matrices $Q_k$ and $B$ are not treatable by either Theorem \ref{T:trB} or Theorem \ref{T:B}.
For $k=2, 3$,  denote by $D(k)$ the local homogeneous dimension, and by 
\begin{align*}
V_{2}(t)=\det \int_{0}^{t}e^{sB}Q_{2}e^{sB^{\star}}, \ \ \ \ \ \ V_{3}(t)=\det \int_{0}^{t}e^{sB}Q_{3}e^{sB^{\star}}ds
\end{align*}
the corresponding volume functions.
\begin{itemize}
\item Let $k=2$. Elementary computations lead to conclude that 
\begin{align*}
D(2)=p_{0}+3p_{1}=2+3 \cdot 2=8,\ \ \ \ \ V_{2}(t) = \frac{t^{4}}{96}(2t^{2}+\cos(2t)-1)\cong t^6,
\end{align*} 
as $t\to \infty$.
In this case, $D = 8>D_\infty = 6$, and therefore \eqref{H0} fails, whereas \eqref{HB} in Hypothesis (B) holds. Therefore, Theorem B applies.
\item Assume $k=3$. We now have
\begin{align*}
D(3)=p_{0}+3p_{1}=3+3=6,\ \ \ \ \ V_{3}(t)=\frac{t^{2}}{96}(12+t^{2})(2t^{2}+\cos(2t)-1)\cong t^6,
\end{align*}
as $t\to \infty$. In such case, $D = 6 = D_\infty$, and therefore \eqref{H0} in Hypothesis (A) holds. In this case, Theorem A applies.
\end{itemize}}

\end{example}

%%%%%%%%%%%%%%%%%%%%%%%%%%%%%%%%%%%%%%%%

%%%%%%%%%%%%%%%%%%%%%%%%%%%%%%%%%%%%%%%%%%%%%%%%%%%%%%%%%%%%%%%%

%%%%%%%%%%%%%%%%%%%%%%%%%%%%%%%%%%%%%%%%%%%%%%%%%%%%%%%%%%%%%%%%%%%%%%

%%%%%%%%%%%%%%%%%%%%%%%%%%%%%%%%%%%%%%%%%%%%%%%%%%%%%%%%%%%

\section{Appendix}\label{S:prelim}

In this section we collect some known results that prove useful in the rest of the paper. First of all, we observe the following simple fact.

\begin{lemma}\label{L:QB}
The following are equivalent:
\begin{itemize}
\item[(i)] $Q$ and $B$ satisfy \emph{\textbf{(H)}};
\item[(ii)] $Q (B^\star)^k x = 0\ \forall k\ge 0\ \Longleftrightarrow\ x = 0$.
\end{itemize}
\end{lemma}

\begin{proof}
We have observed in the introduction that \textbf{(H)}\ $\Longleftrightarrow\ Q(t)>0$ for all $t>0$. It will thus suffice to prove that 
\[
Q(t)>0, \forall t>0\ \Longleftrightarrow\ \text{(ii)}.
\]
Let $x\in \Rn$. We have
\begin{align*}
& \sa Q(t)x,x\da = \int_0^t \sa e^{sB} Q e^{sB^\star} x,x\da ds = \int_0^t |Q^{1/2} e^{sB^\star} x|^2 ds.
\end{align*}
Therefore, $\sa Q(t)x,x\da = 0$ if and only if 
\[
Q e^{sB^\star} x = \sum_{k=0}^\infty \frac{s^k Q (B^\star)^k x}{k!}= 0\ \ \ \forall\ 0\le s\le t.
\]
This is in turn equivalent to 
\[
Q(B^\star)^k x = 0,\ \ \ \forall\ k\ge 0.
\]
This shows $Q(t)>0$ is equivalent to (ii).

\end{proof}

Since in H\"ormander's work \cite{Ho} the class \eqref{A} served as a model for the general class of equations
\[
\sum_{i=1}^m X_i^2 + X_0 = 0,
\]
it seems appropriate to  note here that, having  $Q$ and $B$ in the form \eqref{QB}, is also equivalent to the well-known finite rank condition in \cite{Ho}
\begin{equation}\label{lie}
\operatorname{Rank\ Lie}[Y_0,Y_1,...,Y_n] = n,
\end{equation}
where, with $A = Q^{1/2}$, we have defined 
\begin{equation}\label{Y}
Y_0 u = \sa B x,\nabla u\da,\  \ \ Y_i u = \sum_{j=1}^m a_{ij} \p_{x_j}u,\ i=1,...,m.
\end{equation}
One should be aware that \eqref{QB}, or equivalently \eqref{lie}, does not imply the nilpotency of the Lie algebra $\operatorname{Lie}[Y_0,Y_1,...,Y_m]$ itself. For instance, for the matrices \eqref{ex},
formula \eqref{Y} gives the vector fields 
\[
Y_0 = x\p_x + y\p_z,\ \ \ Y_1 = \p_x,\ \ \ Y_2 = \p_y,
\]
which obviously satisfy \eqref{lie} since $[Y_2,Y_0] = \p_z$. But the Lie algebra is not nilpotent since $[Y_1,Y_0] = Y_1$, and so $[[Y_1,Y_0],Y_0] = Y_1$, etc. A general result in \cite[p.3]{Ka2} states: Let $X_1,...,X_k$ be analytic vector fields  on an analytic manifold $M$, and assume that the Lie algebra $L = \operatorname{Lie}[X_1,...,X_k]$ be nilpotent. If at some $p\in M$ one has $L(p) = T_p M$, then there exists a local system of coordinates around $p$, and a family of \emph{anisotropic dilations} $\delta_\la$ such that, relative to such coordinates, the vector fields $X_1,...,X_k$ have polynomial coefficients, and they are homogeneous of degree $-1$ with respect to $\delta_\la$. It was also shown in the same work that, if $X_1,...,X_k$ are dilation invariant, then Lie$[X_1,...,X_k]$ must be nilpotent.

We observe next that \eqref{mono} implies that $t\to Q(t)$ is monotonically increasing in the sense of quadratic forms (see also Lemma \ref{L:log}), so that, for $Q$ and $B$ satisfying \textbf{(H)}, it always makes sense to consider the formal limit $\underset{t\nearrow \infty}{\lim} Q(t) = Q_\infty$. However, it is not guaranteed that 
\[
Q_\infty = \int_0^\infty e^{sB}Q e^{sB^\star} ds
\]
is a well-defined positive matrix. For instance, for the Kolmogorov's operator in \cite{Kol}, which does satisfy \textbf{(H)}, we have
\[
Q = \begin{pmatrix} 1 & 0\\ 0 & 0\end{pmatrix},\ \ \ \  B = \begin{pmatrix} 0 & 0\\ 1 & 0\end{pmatrix}.
\]
An elementary computation gives for $t>0$ 
\[
Q(t) = \begin{pmatrix} t & \frac{t^2}2\\ \frac{t^2}2 & \frac{t^3}3\end{pmatrix},
\]
which clearly does not converge as $t\to \infty$ to a positive definite matrix. 

A necessary and sufficient condition was found in \cite[Sec.6]{DZ}, where it was proved that $Q_\infty$ exists if and only if the spectrum $\sigma(B)$ of the matrix $B$ satisfies 
\begin{equation}\label{quasi}
\max\{\Re(\lambda)\mid \lambda\in \sigma(B)\}< 0.
\end{equation}
Note that, when \eqref{quasi} is true, then necessarily $\tr B < 0$, and moreover 
\begin{equation}\label{quasinfty}
V(t) = \det Q(t) \ \longrightarrow\ V_\infty = \det Q_\infty > 0.
\end{equation}
For instance, for the Smoluchowski-Kramers' model in \cite{Fre} one has
\begin{equation}\label{skeigenv}
Q = \begin{pmatrix} 1 & 0\\ 0& 0\end{pmatrix},\ \ \ \ \  B = \begin{pmatrix} -2 & -2\\ 1 & 0\end{pmatrix}.
\end{equation} 
The eigenvalues of $B$ are $\la = -1-i$ and $\bar \la = -1+i$, so that \eqref{quasi} is verified. Using \cite[Prop.15.3.2]{Bernie} to find $e^{sB}$, after elementary computations we find
\[
Q(t) = \begin{pmatrix} \frac 14 + \frac{1}4 e^{-2t}(\sin 2t + \cos 2t - 2) & \frac 12 e^{-2t} \sin^2 t \\  \frac 12 e^{-2t} \sin^2 t & \frac 18(1+ e^{-2t}(\sin 2t + \cos 2t - 2))\end{pmatrix}.
\] 
This gives
\[
Q(t)\ \longrightarrow\ Q_\infty = \begin{pmatrix} \frac 14  & 0\\ 0 & \frac 18 \end{pmatrix}.
\]

A basic property, discovered in the work \cite{LP}, is the existence of a homogeneous structure (a family of anisotropic dilations) which locally osculates the intrinsic geometry induced in $\Rn$ by the matrix $Q(t)$. Since such local structure plays a critical role in the Strichartz estimates in Theorems A and B, we next recall some of its salient aspects. 
From the results in \cite{Ku, Ku2}, \cite{LP}, we know that a necessary and sufficient condition for the invariance of \eqref{A} with respect to the non-isotropic dilations $\delta_\la$ in \eqref{sdl} is that $B$ takes the special form \eqref{barB}. 
The fact that $\delta_\lambda$ are group automorphisms with respect to the Lie group law \eqref{Lie} follows from the commutation property established in \cite[p.288]{Ku2}, 
\begin{equation}\label{expdil}
e^{-\lambda^2 t \bar B}=\delta_\lambda e^{-t \bar B}\delta_{\lambda^{-1}},
\end{equation}
valid for any $\lambda>0$ and $t\in\R$.
We denote by 
\[
\overline V(t) = \det \overline Q(t) = \det \int_0^t e^{s\overline B} Q e^{s \overline B^\star} ds,
\]
the volume function generated by $Q$ and $\overline B$. From \eqref{expdil}, and the fact that from \eqref{barB} we have $\tr \bar B = 0$, one infers that 
\begin{equation}\label{voldil}
\overline V(t) = \gamma_0\ t^{D},\ \ \ \ \ t>0,
\end{equation}
where $D$ is given by \eqref{hd}, and $\gamma_0 = \overline V(1) > 0$, see \eqref{vol}. 

If now $Q$ and $B$ are in the canonical form \eqref{QB}, and $\overline B$ is obtained from $B$ as in \eqref{barB}, then \cite[(3.14) and Remark 3.1]{LP} imply the following result.

\begin{proposition}\label{P:asy}
One has
$\underset{t\to 0^+}{\lim}\ \frac{V(t)}{\overline V(t)} = 1$.
\end{proposition}
Proposition \ref{P:asy} and \eqref{voldil} allow to conclude that there exists $\gamma>0$ such that  
\begin{equation}\label{vicinoa0}
V(t) \ge \gamma t^D,\ \ \ \ \ \ 0<t\le 1.
\end{equation}
As $t\to \infty$, the function $V(t)$ exhibits a highly variable behavior, which need not follow a power-law form (see the example in \eqref{skeigenv}).
In the next result we have combined \cite[Prop. 3.1]{GTma} with \cite[Prop. 2.3]{BGT}.

\begin{proposition}\label{P:boom0}
The following is true:
\begin{itemize}
\item[(i)] If $\tr B \ge 0$, then there exists a constant $c_1>0$ such that $V(t) \geq c_1 t^2$ for all $t\geq 1$.
\item[(ii)] If $\tr B \ge 0$ and furthermore $\max\{\Re(\lambda)\mid \lambda\in \sigma(B)\}=L_0>0$, then there exists a constant $c_0$ such that $V(t)\geq c_0 e^{2 L_0 t}$ for all $t\geq 1$.
\item[(iii)] Finally, if $\max\{\Re(\lambda)\mid \lambda\in \sigma(B)\}\geq 0$, then there exists a constant $c_0>0$ such that $V(t)\geq c_0 t$, for every $t\ge 1$.
\end{itemize}
\end{proposition}

%%%%%%%%%%%%%%%%%%%%%%%%%%%%%%%%%%%%%%%%%%%%%%%%%%%%%%%%%%%%%

\bibliographystyle{amsplain}

\begin{thebibliography}{10}


\bibitem{BG}
H. Bahouri \& I. Gallagher, \emph{Local dispersive and Strichartz estimates for the Schr\"odinger operator on the Heisenberg group},
Commun. Math. Res. 39~(2023), no. 1, 1-35.


\bibitem{BBG}
H. Bahouri. D. Barilari \& I. Gallagher, \emph{ Strichartz estimates and Fourier restriction theorems on the Heisenberg group}.
J. Fourier Anal. Appl. 27~(2021), no. 2, Paper No. 21, 41 pp.

\bibitem{BFG}
H. Bahouri, C. Fermanian-Kammerer \& I. Gallagher, \emph{Dispersive estimates for the Schr\"odinger operator on step-$2$ stratified Lie groups}.
Anal. PDE 9~(2016), no. 3, 545-574.

\bibitem{BF}
D. Barilari \& S. Flynn, \emph{Refined Strichartz Estimates for sub-Laplacians in Heisenberg and H-type groups}. ArXiv:2501.04415

\bibitem{BS}
C. Bennett \& and R. C. Sharpley, \emph{Interpolation of operators}. Vol. 129. Academic press, 1988.

\bibitem{Bernie}
D. S. Bernstein, \emph{Matrix mathematics}.
Princeton University Press, Princeton, NJ, 2005, xxxviii+726 pp.

\bibitem{Bri}
H. C. Brinkman, \emph{Brownian motion in a field of force and the diffusion theory of chemical reactions. II}. Physica 23~(1956), 149-155.

\bibitem{BGT}
F. Buseghin, N. Garofalo \& G. Tralli, \emph{On the limiting behaviour of some nonlocal seminorms: a new phenomenon}.
Ann. Sc. Norm. Super. Pisa Cl. Sci. (5) 23~(2022), no. 2, 837-875.


\bibitem{Caze}
 T. Cazenave, \emph{Semilinear Schr\"odinger equations}. Courant Lect. Notes Math., 10
New York University, Courant Institute of Mathematical Sciences, New York; American Mathematical Society, Providence, RI, 2003, xiv+323 pp.

\bibitem{CDKR}
M. Cowling, A. H. Dooley, A. Kor\'anyi \& F. Ricci, \emph{$H$-type groups and Iwasawa decompositions}. Adv. Math. 87~(1991), no. 1, 1-41.


\bibitem{DZ}
G. Da Prato and J. Zabczyk, \emph{Ergodicity for infinite-dimensional systems}, London Mathematical Society Lecture Note Series \textbf{229}~(1996), Cambridge University Press, Cambridge.

\bibitem{Dym}
H. Dym, 
\emph{Stationary measures for the flow of a linear differential equation driven by white noise}.
Trans. Amer. Math. Soc. 123~(1966), 130-164.

\bibitem{Fo}
G. B. Folland, \emph{Harmonic analysis in phase space}.
Ann. of Math. Stud., 122
Princeton University Press, Princeton, NJ, 1989, x+277 pp.


\bibitem{FS}
G. B. Folland \& E. M. Stein, \emph{Hardy spaces on homogeneous groups}.
Math. Notes, 28
Princeton University Press, Princeton, NJ; University of Tokyo Press, Tokyo, 1982, xii+285 pp.

\bibitem{Fre}
M. Freidlin, \emph{Some remarks on the Smoluchowski-Kramers approximation}. J. Statist. Phys. 117~(2004), no. 3-4, 617-634.

\bibitem{GaXu}
J. Gallier \& D. Xu, \emph{Computing exponentials of skew-symmetric matrices and logarithms of orthogonal matrices}. International Journal of Robotics and Automation, 17~(2002), no. 4, 1-11. 

%\bibitem{G}
%N. Garofalo, \emph{Hypoelliptic operators and some aspects of analysis and geometry of sub-Riemannian spaces}.
%EMS Ser. Lect. Math.
%European Mathematical Society (EMS), Z\"urich, 2016, 123-257.

\bibitem{GL}
N. Garofalo \& A. Lunardi, \emph{Schr\"odinger semigroups and the H\"ormander hypoellipticity condition}, Math. Z. 312~(2026), no. 1, Paper No. 2, 16 pp.   

\bibitem{GTma}
N. Garofalo \& G. Tralli, \emph{Hardy-Littlewood-Sobolev inequalities for a class of non-symmetric and non-doubling hypoelliptic semigroups}.
Math. Ann. 383~(2022), no. 1-2, 1-38.

\bibitem{GK}
Y. Giga \& R. V. Kohn, \emph{
 Asymptotically self-similar blow-up of semilinear heat equations}.
Comm. Pure Appl. Math. 38 (1985), no. 3, 297-319.



\bibitem{GV}
J. Ginibre \& G. Velo, \emph{On a class of nonlinear Schr\"odinger equations. I. The Cauchy problem, general case}.
J. Functional Analysis 32 (1979), no. 1, 1-32.

\bibitem{GVstrich}
J. Ginibre \& G. Velo, \emph{The global Cauchy problem for the nonlinear Schr\"odinger equation revisited}.
Ann. Inst. H. Poincar\'e Anal. Non Lin\'eaire 2~(1985), no. 4, 309-327.





\bibitem{Ho}
L. H{\"o}rmander,
\textit{Hypoelliptic second order differential equations}. Acta Math. 119~(1967), 147-171.

\bibitem{Hobook}
L. H{\"o}rmander,
\textit{The analysis of linear partial differential operators. I}.
Classics Math.
Springer-Verlag, Berlin, 2003, x+440 pp.

\bibitem{HJ}
R. A. Horn \& C. R. Johnson, \emph{Matrix analysis}.
Cambridge University Press, Cambridge, 2013, xviii+643 pp.

\bibitem{HMMS}
D. Hundertmark, M. Meyries, L. Machinek \& R. Schnaubelt, \emph{Operator semigroups and dispersive equations}, 16th Internet Seminar on Evolution Equations, 2013.

\bibitem{Iomrl}
A. D. Ionescu, \emph{A maximal operator and a covering lemma on non-compact symmetric spaces}.
Math. Res. Lett. 7 (2000), no. 1, 83-93.

\bibitem{Iojfa}
A. D. Ionescu, \emph{Fourier integral operators on noncompact symmetric spaces of real rank one}.
J. Funct. Anal. 174 (2000), no. 2, 274-300.

\bibitem{Iojfa2}
A. D. Ionescu, \emph{On the Poisson transform on symmetric spaces of real rank one}.
J. Funct. Anal. 174 (2000), no. 2, 513-523.

\bibitem{Ioann}
A. D. Ionescu, \emph{An endpoint estimate for the Kunze-Stein phenomenon and related maximal operators}.
Ann. of Math. (2) 152 (2000), no. 1, 259-275.


\bibitem{Ioduke}
A. D. Ionescu, \emph{Singular integrals on symmetric spaces of real rank one}.
Duke Math. J. 114 (2002), no. 1, 101-122.


 

\bibitem{Kal}
R. E. Kalman, \emph{Contributions to the theory of optimal control}.
Bol. Soc. Mat. Mexicana (2) 5~(1960), 102-119.

\bibitem{Ka}
A. Kaplan, \emph{Fundamental solutions for a class of hypoelliptic PDE generated by composition of quadratic forms}. Trans. Amer. Math. Soc. 258~(1980), no. 1, 147-153.


\bibitem{Ka2}
M. Kawski, \emph{Nilpotent Lie algebras of vector fields}.
J. Reine Angew. Math. 388~(1988), 1-17.

\bibitem{KT}
M. Keel \& T. Tao, \emph{Endpoint Strichartz estimates}.
Amer. J. Math. 120~(1998), no. 5, 955-980.

\bibitem{Kol}
A. N. Kolmogorov,  
\textit{Zuf\"allige Bewegungen (Zur Theorie der Brownschen Bewegung)}. Ann. of Math. (2) 35~(1934), 116-117.

\bibitem{Ku}
L. P. Kupcov, \emph{The fundamental solutions of a certain class of elliptic-parabolic second order equations}.
Differencial'nye Uravnenija 8~(1972), 1649-1660.

\bibitem{Ku2}
L. P. Kupcov, \emph{Fundamental solutions of some second-order degenerate parabolic equations}.
Mat. Zametki 31~(1982), no. 4, 559-570.

\bibitem{LP}
E. Lanconelli \& S. Polidoro, \emph{On a class of hypoelliptic evolution operators}. Partial differential equations, II (Turin, 1993). Rend. Sem. Mat. Univ. Politec. Torino \textbf{52}~(1994), no. 1, 29-63.

\bibitem{McC}
J.M. McCarthy, \emph{Introduction to theoretical kinematics}. Cambridge, MA, MIT Press, 1990. 

\bibitem{MR}
F. Merle \& P. Rapha\"el, \emph{Sharp upper bound on the blow-up rate for the critical nonlinear Schr\"odinger equation}. Geom. funct. anal.
Vol. 13 (2003) 591 – 642

\bibitem{MR2}
F. Merle \& P. Rapha\"el, \emph{Profiles and quantization of the blow up mass for critical nonlinear Schr\"odinger equation}. Comm. Math. Phys. 253, 675-704 (2005)

\bibitem{MR3}
F. Merle \& P. Rapha\"el, \emph{On a sharp lower bound on the blow-up rate for the $L^{2}$ critical nonlinear Schr\"odinger equation}. J. of the American Mathematical Society,
Volume 19, Number 1, Pages 37-90.

\bibitem{MRS2} F. Merle, P. Rapha\"el \& J. Szeftel, \emph{Stable self-similar blow-up dynamics for slightly $L^{2}$ super-critical NLS equations}, Geom. Funct. Anal. Vol. 20 (2010) 1028-1071.

\bibitem{MRS}
F. Merle, P. Rapha\"el \& J. Szeftel, \emph{On collapsing ring blow-up solutions to the mass supercritical nonlinear Schr\"odinger equation}.	Duke Math. J.,
Vol. 163, No. 2, 2012.

\bibitem{Mu}
D. M\"uller, \emph{A restriction theorem for the Heisenberg group}.
Ann. of Math. (2) 131~(1990), no. 3, 567-587.

\bibitem{MuLiSa}
R.M. Murray, Z.X. Li, \& S.S. Sastry,  \emph{A Mathematical introduction to robotics manipulation}. London, UK, CRC Press, 1994.  

\bibitem{NSW}
A. Nagel, E. M. Stein \& S.Wainger, \emph{Balls and metrics defined by vector fields. I. Basic properties}.
Acta Math. 155~(1985), no. 1-2, 103-147.

\bibitem{O}
R. O'Neil, \emph{Convolution operators and  $L(p,q)$ spaces}.
Duke Math. J. 30~(1963), 129-142.

\bibitem{OU}
L. S. Ornstein \& G. E. Uhlenbeck, \emph{On the theory of the Brownian motion. I}.
Phys. Rev. (2) 36~(1930), 823-841. 

\bibitem{RS}
L. P. Rothschild \& E. M. Stein, \emph{Hypoelliptic differential operators and nilpotent groups}. Acta Math. \textbf{137}~(1976), no. 3-4, 247-320. 

\bibitem{Sogge}
C. D. Sogge,  \emph{Lectures on nonlinear wave equations}.
Monogr. Anal., II
International Press, Boston, MA, 1995, vi+159 pp.

\bibitem{St}
E. M. Stein, \emph{Singular integrals and differentiability properties of functions}.
Princeton Math. Ser., No. 30
Princeton University Press, Princeton, NJ, 1970, xiv+290 pp.


\bibitem{Stri}
R. S. Strichartz, \emph{Restrictions of Fourier transforms to quadratic surfaces and decay of solutions of wave equations}.
Duke Math. J. 44~(1977), no. 3, 705-714.

\bibitem{Strjfa}
R. S. Strichartz, \emph{Harmonic analysis as spectral theory of Laplacians}.
J. Funct. Anal. 87 (1989), no. 1, 51-148.

\bibitem{Su}
C. Sulem \& P.-L. Sulem, \emph{The nonlinear Schr\"odinger equation}.
Appl. Math. Sci., 139
Springer-Verlag, New York, 1999, xvi+350 pp.

\bibitem{Tao}
T. Tao, \emph{Nonlinear dispersive equations}.
CBMS Reg. Conf. Ser. Math., 106

\bibitem{Veluma}
S. Thangavelu, \emph{Harmonic analysis on the Heisenberg group}.
Progr. Math., 159
Birkh\"auser Boston, Inc., Boston, MA, 1998, xiv+192 pp.

\bibitem{We}
F. B. Weissler, \emph{Single point blow-up for a semilinear initial value problem}.
J. Differential Equations 55 (1984), no. 2, 204-224.

\bibitem{Z}
J. Zabczyk, \emph{Mathematical control theory: an introduction}.
Systems Control Found. Appl.
Birkh\"auser Boston, Inc., Boston, MA, 1992, x+260 pp.




\end{thebibliography}

\end{document}